\documentclass[10pt]{article}
\usepackage{latexsym,amssymb,amsmath}
\usepackage[notref,notcite]{showkeys}
\usepackage{amssymb,amsthm,upref,amscd}
\usepackage{color}
\usepackage[T1]{fontenc}
\begin{document}
\baselineskip=14pt

\numberwithin{equation}{section}

\newtheorem{thm}{Theorem}[section]
\newtheorem{lem}[thm]{Lemma}
\newtheorem{cor}[thm]{Corollary}
\newtheorem{Prop}[thm]{Proposition}
\newtheorem{Def}[thm]{Definition}
\newtheorem{Rem}[thm]{Remark}
\newtheorem{Ex}[thm]{Example}

\newcommand{\A}{\mathbb{A}}
\newcommand{\B}{\mathbb{B}}
\newcommand{\C}{\mathbb{C}}
\newcommand{\D}{\mathbb{D}}
\newcommand{\E}{\mathbb{E}}
\newcommand{\F}{\mathbb{F}}
\newcommand{\G}{\mathbb{G}}
\newcommand{\I}{\mathbb{I}}
\newcommand{\J}{\mathbb{J}}
\newcommand{\K}{\mathbb{K}}
\newcommand{\M}{\mathbb{M}}
\newcommand{\N}{\mathbb{N}}
\newcommand{\Q}{\mathbb{Q}}
\newcommand{\R}{\mathbb{R}}
\newcommand{\T}{\mathbb{T}}
\newcommand{\U}{\mathbb{U}}
\newcommand{\V}{\mathbb{V}}
\newcommand{\W}{\mathbb{W}}
\newcommand{\X}{\mathbb{X}}
\newcommand{\Y}{\mathbb{Y}}
\newcommand{\Z}{\mathbb{Z}}
\newcommand\ca{\mathcal{A}}
\newcommand\cb{\mathcal{B}}
\newcommand\cc{\mathcal{C}}
\newcommand\cd{\mathcal{D}}
\newcommand\ce{\mathcal{E}}
\newcommand\cf{\mathcal{F}}
\newcommand\cg{\mathcal{G}}
\newcommand\ch{\mathcal{H}}
\newcommand\ci{\mathcal{I}}
\newcommand\cj{\mathcal{J}}
\newcommand\ck{\mathcal{K}}
\newcommand\cl{\mathcal{L}}
\newcommand\cm{\mathcal{M}}
\newcommand\cn{\mathcal{N}}
\newcommand\co{\mathcal{O}}
\newcommand\cp{\mathcal{P}}
\newcommand\cq{\mathcal{Q}}
\newcommand\rr{\mathcal{R}}
\newcommand\cs{\mathcal{S}}
\newcommand\ct{\mathcal{T}}
\newcommand\cu{\mathcal{U}}
\newcommand\cv{\mathcal{V}}
\newcommand\cw{\mathcal{W}}
\newcommand\cx{\mathcal{X}}
\newcommand\ocd{\overline{\cd}}

\def\c{\centerline}
\def\ov{\overline}
\def\emp {\emptyset}
\def\pa {\partial}
\def\bl{\setminus}
\def\op{\oplus}
\def\sbt{\subset}
\def\un{\underline}
\def\al {\alpha}
\def\bt {\beta}
\def\de {\delta}
\def\Ga {\Gamma}
\def\ga {\gamma}
\def\lm {\lambda}
\def\Lam {\Lambda}
\def\om {\omega}
\def\Om {\Omega}
\def\sa {\sigma}
\def\vr {\varepsilon}
\def\va {\varphi}

\title{\bf On the Brezis-Nirenberg type critical problem for nonlinear Choquard equation\thanks{Partially supported by NSFC (11101374, 11271331, 11571317) and ZJNSF(LY15A010010)}}

\author{ Fashun Gao, Minbo Yang\thanks{M. Yang is the corresponding author: mbyang@zjnu.edu.cn}
\\
\\
{\small Department of Mathematics, Zhejiang Normal University} \\ {\small  Jinhua, Zhejiang, 321004, P. R. China}}

\date{}
\maketitle

\begin{abstract}
We establish some existence results for the  Brezis-Nirenberg type problem of the nonlinear Choquard equation
$$-\Delta u
=\left(\int_{\Omega}\frac{|u|^{2_{\mu}^{\ast}}}{|x-y|^{\mu}}dy\right)|u|^{2_{\mu}^{\ast}-2}u+\lambda u\hspace{4.14mm}\mbox{in}\hspace{1.14mm} \Omega,
$$
where $\Omega$ is a bounded domain of $\mathbb{R}^N$ with Lipschitz boundary, $\lambda$ is a real parameter, $N\geq3$, $2_{\mu}^{\ast}=(2N-\mu)/(N-2)$ is the critical
exponent in the sense of the Hardy-Littlewood-Sobolev inequality.
 \vspace{0.3cm}


\vspace{0.3cm}

\end{abstract}

\section{Introduction and main results}
In the last decades many people studied the elliptic equation
\begin{equation}\label{local.S1}
\left\{\begin{array}{l}
\displaystyle-\Delta u=|u|^{2^{\ast}-2}u+\lambda u\ \ \mbox{in}\ \ \Omega,\\
\\
\displaystyle u=0 \hspace{25.14mm}\ \ \mbox{on}\ \ \partial\Omega,
\end{array}
\right.
\end{equation}
where $\Omega$ is a bounded domain of $\R^N$, $2^{\ast}=\frac{2N}{N-2}$ is the critical exponent for the embedding of $H_{0}^{1}(\Omega)$ to $L^p(\Omega)$, $ \lambda\in (0,\lambda_{1})$ where $\lambda_{1}$ is the first eigenvalue of $-\Delta$ set on bounded domain. In a celebrated paper \cite{BN} Brezis and Nirenberg proved that: if $N\geq4$ and $\lambda\in(0,\lambda_{1})$, then problem \eqref{local.S1} has a nontrivial solution; if $N=3$ then there exists a constant $\lambda_{\ast}\in(0, \lambda_{1})$ such that for any $\lambda\in(\lambda_{\ast},\lambda_{1})$ problem \eqref{local.S1} has a positive solution and if $\Omega$ is a ball, problem \eqref{local.S1} has a positive solution if and only if $\lambda\in(\frac{\lambda_{1}}{4},\lambda_{1})$. Capozzi, Fortunato and Palmieri \cite{CFP} proved if $N\geq4$ then the problem \eqref{local.S1} has a nontrivial solution for all $\lambda>0$. In \cite{CSS}, Cerami, Solimini and Struwe proved if $N\geq6$ and $\lambda\in(0,\lambda_{1})$, the existence of
sign-changing solutions; if $\Omega$ is a ball, $N\geq7$ and $\lambda\in(0,\lambda_{1})$, infinitely many radial solutions to problem \eqref{local.S1}. There is a great deal of work on elliptic equations with critical nonlinearity, see for example \cite{CP, CFS, FG, GY, J,  Rph1, SZ, Wi} and the references therein. 

In the present paper we are going to consider the existence and nonexistence of solutions for the following nonlocal equation:
\begin{equation}\label{CCE}
\left\{\begin{array}{l}
\displaystyle-\Delta u
=\left(\int_{\Omega}\frac{|u|^{2_{\mu}^{\ast}}}{|x-y|^{\mu}}dy\right)|u|^{2_{\mu}^{\ast}-2}u+\lambda u\hspace{4.14mm}\mbox{in}\hspace{1.14mm} \Omega,\\
\displaystyle u\in H_{0}^{1}(\Omega),
\end{array}
\right.
\end{equation}
where $\Omega$ is a bounded domain of $\mathbb{R}^N$ with Lipschitz boundary, $\lambda$ is a real parameter, $N\geq3$, $0<\mu<N$ and $2_{\mu}^{\ast}=(2N-\mu)/(N-2)$.  This nonlocal elliptic equation is closely related to the nonlinear Choquard equation
\begin{equation}\label{Nonlocal.S1}
 -\Delta u +V(x)u =\Big(\frac{1}{|x|^{\mu}}\ast|u|^{p}\Big)|u|^{p-2}u  \quad \mbox{in} \quad \R^3.
\end{equation}
Different from the fractional Laplacian where the pseudo-differential operator causes the nonlocal phenomena, for the Choquard equation the nonlocal term appears in the nonlinearity and influences the equation greatly. For $p=2$ and $\mu=1$, it  goes
back to the description of the quantum theory of a polaron at rest by S. Pekar in 1954 \cite{P1}
and the modeling of an electron trapped
in its own hole in 1976 in the work of P. Choquard, as a certain approximation to Hartree-Fock theory of one-component
plasma \cite{L1}. In some particular cases, this equation is also known as the Schr\"{o}dinger-Newton equation, which was introduced by Penrose in his discussion on the selfgravitational collapse of a quantum mechanical wave function  \cite{Pe}.

The existence and qualitative properties of solutions of \eqref{Nonlocal.S1} have been widely studied in the last decades. In \cite{L1}, Lieb proved the existence and uniqueness, up to translations,
of the ground state. Later, in \cite{Ls}, Lions showed the existence of
a sequence of radially symmetric solutions. In \cite{CCS1, ML,  MS1} the authors showed the regularity, positivity
and radial symmetry of the ground states and
derived decay property at infinity as well. Moreover, Moroz and Van
Schaftingen in \cite{MS2} considered  the existence of ground states under the assumptions of Berestycki-Lions type. For periodic potential $V$ that changes sign and $0$ lies in the gap of the spectrum of the Schr\"{o}dinger operator $-\Delta +V$, the problem is strongly indefinite, and the existence of solution for $p=2$ was considered in \cite{BJS} by reduction arguments.  In \cite{ANY} Alves, N\'obrega and the second author studied the existence of multi-bump shaped solution for the nonlinear Choquard equation with deepening potential well. For a general case, Ackermann \cite{AC} proposed a new approach to prove the existence of infinitely many geometrically distinct weak solutions. For other related results, we refer the readers to \cite{CS, GS} for the existence of sign-changing solutions, \cite{AY1, AY2, MS3, S, WW, YD} for the existence and concentration behavior of the semiclassical solutions.

The starting point of the variational approach to the problem \eqref{CCE} is the following  well-known Hardy-Littlewood-Sobolev inequality.
\begin{Prop}\label{HLS}
 (Hardy-Littlewood-Sobolev inequality). (See \cite{LL}.) Let $t,r>1$ and $0<\mu<N$ with $1/t+\mu/N+1/r=2$, $f\in L^{t}(\mathbb{R}^N)$ and $h\in L^{r}(\mathbb{R}^N)$. There exists a sharp constant $C(t,N,\mu,r)$, independent of $f,h$, such that
\begin{equation}\label{HLS1}
\int_{\mathbb{R}^{N}}\int_{\mathbb{R}^{N}}\frac{f(x)h(y)}{|x-y|^{\mu}}dxdy\leq C(t,N,\mu,r) |f|_{t}|h|_{r}.
\end{equation}
If $t=r=2N/(2N-\mu)$, then
$$
 C(t,N,\mu,r)=C(N,\mu)=\pi^{\frac{\mu}{2}}\frac{\Gamma(\frac{N}{2}-\frac{\mu}{2})}{\Gamma(N-\frac{\mu}{2})}\left\{\frac{\Gamma(\frac{N}{2})}{\Gamma(N)}\right\}^{-1+\frac{\mu}{N}}.
$$
In this case there is equality in \eqref{HLS1} if and only if $f\equiv(const.)h$ and
$$
h(x)=A(\gamma^{2}+|x-a|^{2})^{-(2N-\mu)/2}
$$
for some $A\in \mathbb{C}$, $0\neq\gamma\in\mathbb{R}$ and $a\in \mathbb{R}^{N}$.
\end{Prop}

Notice that, by the Hardy-Littlewood-Sobolev inequality, the integral
$$
\int_{\mathbb{R}^{N}}\int_{\mathbb{R}^{N}}\frac{|u(x)|^{q}|u(y)|^{q}}{|x-y|^{\mu}}dxdy
$$
is well defined if $|u|^{q}\in L^{t}(\mathbb{R}^N)$ for some $t>1$ satisfying
$$
\frac{2}{t}+\frac{\mu}{N}=2.
$$
Thus, for $u\in H^{1}(\mathbb{R}^N)$,  by Sobolev embedding Theorems, we know
$$
2\leq tq\leq \frac{2N}{N-2},
$$
that is
$$
\frac{2N-\mu}{N}\leq q\leq\frac{2N-\mu}{N-2}.
$$
Thus, $\frac{2N-\mu}{N}$ is called the lower critical exponent and $2_{\mu}^{\ast}=\frac{2N-\mu}{N-2}$ is the upper critical exponent in the sense of the Hardy-Littlewood-Sobolev inequality.

We need to point out that all the papers we mentioned above  considered the nonlinear Choquard equation with superlinear subcritical nonlinearities. In a recent paper \cite{MS4} by Moroz and Van
Schaftingen, the authors considered the nonlinear Choquard equation \eqref{Nonlocal.S1}  in $\R^N$ with lower critical exponent $\frac{2N-\mu}{N}$. There the authors investigated the existence and nonexistence of solutions to the equation
with nonconstant potential by minimizing arguments. However, as far as we know there seems no result for the  nonlinear Choquard equation with upper critical exponent with respect to  the Hardy-Littlewood-Sobolev inequality. In \cite{ACTY}, the authors studied the existence and concentrations of the solutions of a nonlocal Schr\"{o}dinger with the critical exponential growth in $\R^2$, that problem is closely related to the Choquard equation. Recently many people also studied the Brezis-Nirenberg problem for elliptic equation driven by the fractional Laplacian, this type of problem are nonlocal in nature and we may refer the readers to \cite{BCSS, SV2, Ta} and the references therein for a recent progress. And so, it is quite natural to ask if the well-known results established by Brezis and Nirenberg in \cite{BN} for local elliptic equation still hold for the nonlocal Choquard equation. The main purpose of the present paper is to study the nonlinear Choquard equation with upper critical exponent $2_{\mu}^{\ast}=\frac{2N-\mu}{N-2}$ and give a confirm answer to the question of the existence and nonexistence of solutions.

From the Hardy-Littlewood-Sobolev inequality, for all $u\in D^{1,2}(\mathbb{R}^N)$ we know
$$
\Big(\int_{\mathbb{R}^N}\int_{\mathbb{R}^N}\frac{|u(x)|^{2_{\mu}^{\ast}}|u(y)|^{2_{\mu}^{\ast}}}{|x-y|^{\mu}}dxdy\Big)^{\frac{N-2}{2N-\mu}}\leq C(N,\mu)^{\frac{N-2}{2N-\mu}}|u|_{2^{\ast}}^{2},
$$
where $C(N,\mu)$ is defined as in the Proposition \ref{HLS}.
We use $S_{H,L}$ to denote best constant defined by
\begin{equation}\label{S1}
S_{H,L}:=\displaystyle\inf\limits_{u\in D^{1,2}(\mathbb{R}^N)\backslash\{{0}\}}\ \ \frac{\displaystyle\int_{\mathbb{R}^N}|\nabla u|^{2}dx}{(\displaystyle\int_{\mathbb{R}^N}\int_{\mathbb{R}^N}\frac{|u(x)|^{2_{\mu}^{\ast}}|u(y)|^{2_{\mu}^{\ast}}}{|x-y|^{\mu}}dxdy)^{\frac{N-2}{2N-\mu}}}.
\end{equation}

From commentaries above, we can easily draw the following conclusion.
\begin{lem}\label{ExFu}
The constant $S_{H,L}$ defined in \eqref{S1} is achieved if and only if $$u=C\left(\frac{b}{b^{2}+|x-a|^{2}}\right)^{\frac{N-2}{2}} ,$$ where $C>0$ is a fixed constant, $a\in \mathbb{R}^{N}$ and $b\in(0,\infty)$ are parameters. What's more,
$$
S_{H,L}=\frac{S}{C(N,\mu)^{\frac{N-2}{2N-\mu}}},
$$
where $S$ is the best Sobolev constant.
\end{lem}
\begin{proof}
By the Hardy-Littlewood-Sobolev inequality, we can see
$$\aligned
S_{H,L}\geq\frac{1}{C(N,\mu)^{\frac{N-2}{2N-\mu}}}\inf\limits_{u\in D^{1,2}(\mathbb{R}^N)\backslash\{{0}\}}\ \ \frac{\displaystyle\int_{\mathbb{R}^N}|\nabla u|^{2}dx}{|u|_{2^{\ast}}^{2}}
=\frac{S}{C(N,\mu)^{\frac{N-2}{2N-\mu}}},
\endaligned
$$
where $S$ is the best Sobolev constant. Notice that  the equality in the Hardy-Littlewood-Sobolev inequality holds  if and only if $u=C\left(\frac{b}{b^{2}+|x-a|^{2}}\right)^{\frac{N-2}{2}} $, where $C>0$ is a fixed constant, $a\in \mathbb{R}^{N}$ and $b\in(0,\infty)$ are parameters. Meanwhile, it is well-known that the function $u=C\left(\frac{b}{b^{2}+|x-a|^{2}}\right)^{\frac{N-2}{2}} $ is also a minimizer for $S$, thus we get that $S_{H,L}$ is achieved if and only if $u=C\left(\frac{b}{b^{2}+|x-a|^{2}}\right)^{\frac{N-2}{2}} $ and$$S_{H,L}=\frac{S}{C(N,\mu)^{\frac{N-2}{2N-\mu}}}.$$
\\
In particular, let $U(x):=\frac{[N(N-2)]^{\frac{N-2}{4}}}{(1+|x|^{2})^{\frac{N-2}{2}}}$ be a minimizer for $S$, then
\begin{equation}\label{REL}
\aligned
\tilde{U}(x)&=S^{\frac{(N-\mu)(2-N)}{4(N-\mu+2)}}C(N,\mu)^{\frac{2-N}{2(N-\mu+2)}}U(x)\\
&=S^{\frac{(N-\mu)(2-N)}{4(N-\mu+2)}}C(N,\mu)^{\frac{2-N}{2(N-\mu+2)}}\frac{[N(N-2)]^{\frac{N-2}{4}}}{(1+|x|^{2})^{\frac{N-2}{2}}}
\endaligned
\end{equation}
is the unique  minimizer for $S_{H,L}$ and satisfies
$$
-\Delta u=\left(\int_{\R^N}\frac{|u|^{2_{\mu}^{\ast}}}{|x-y|^{\mu}}dy\right)|u|^{2_{\mu}^{\ast}-2}u\ \ \   \hbox{in}\ \ \ \R^N.
$$
Moreover,
$$
\int_{\mathbb{R}^N}|\nabla \tilde{U}|^{2}dx=\int_{\mathbb{R}^N}\int_{\mathbb{R}^N}\frac{|\tilde{U}(x)|^{2_{\mu}^{\ast}}|\tilde{U}(y)|^{2_{\mu}^{\ast}}}{|x-y|^{\mu}}dxdy=S_{H,L}^{\frac{2N-\mu}{N-\mu+2}}.
$$
\end{proof}
We have some more words about the best constant $S_{H,L}$.
\begin{lem}
Let $N\geq3$. For every open subset $\Omega$ of $\mathbb{R}^N$,
\begin{equation}
S_{H,L}(\Omega):=\displaystyle\inf\limits_{u\in D_{0}^{1,2}(\Omega)\backslash\{{0}\}}\ \ \frac{\displaystyle\int_{\Omega}|\nabla u|^{2}dx}{\left(\displaystyle\int_{\Omega}\int_{\Omega}\frac{|u(x)|^{2_{\mu}^{\ast}}|u(y)|^{2_{\mu}^{\ast}}}{|x-y|^{\mu}}dxdy\right)^{\frac{N-2}{2N-\mu}}}=S_{H,L},
\end{equation}
$S_{H,L}(\Omega)$ is never achieved except when $\Omega=\R^N$.
\end{lem}
\begin{proof} It is clear that $S_{H,L}\leq S_{H,L}(\Omega)$ by $D_{0}^{1,2}(\Omega)\subset D^{1,2}(\mathbb{R}^N)$. Let $\{u_{n}\}\subset C_{0}^{\infty}(\mathbb{R}^N)$ be a minimizing sequence for $S_{H,L}$. We make translations and dilations for $\{u_{n}\}$ by choosing $y_{n}\in\mathbb{R}^N$ and $\tau_{n}>0$ such that
$$
u_{n}^{y_{n},\tau_{n}}(x):=\tau_n^{\frac{N-2}{2}}u_{n}(\tau_{n} x+y_{n})\in C_{0}^{\infty}(\Omega),
$$
which satisfies
$$
\int_{\mathbb{R}^N}|\nabla u_{n}^{y_{n},\tau_{n}}|^{2}dx=\int_{\mathbb{R}^N}|\nabla u_{n}|^{2}dx
$$
and
$$
\int_{\Omega}\int_{\Omega}\frac{|u_{n}^{y_{n},\tau_{n}}(x)|^{2_{\mu}^{\ast}}|u_{n}^{y_{n},\tau_{n}}(y)|^{2_{\mu}^{\ast}}}{|x-y|^{\mu}}dxdy
=\int_{\mathbb{R}^N}\int_{\mathbb{R}^N}\frac{|u_{n}(x)|^{2_{\mu}^{\ast}}|u_{n}(y)|^{2_{\mu}^{\ast}}}{|x-y|^{\mu}}dxdy.
$$
Hence we obtain $S_{H,L}(\Omega)\leq S_{H,L}$.   $S_{H,L}(\Omega)$ is never achieved except when $\Omega=\R^N$ is due to the fact that $\tilde{U}(x)$ is the only class of functions such that the equality holds in the  Hardy-Littlewood-Sobolev inequality and attains the best constant.
\end{proof}

Next we will denote the sequence of eigenvalues of the operator $-\Delta$ on $\Omega$ with homogeneous Dirichlet boundary data by
$$
0<\lambda_{1}<\lambda_{2}\leq...\leq \lambda_{j}\leq\lambda_{j+1}\leq...
$$
and
$$
\lambda_{j}\rightarrow+\infty
$$
as $j\rightarrow+\infty$. Moreover, $\{e_{j}\}_{j\in\mathbb{N}}\subset L^{\infty}(\Omega)$ will be the sequence of eigenfunctions corresponding to $\lambda_{j}$. We recall that this sequence is an orthonormal basis of $L^{2}(\Omega)$ and an orthogonal basis of $H_{0}^{1}(\Omega)$. We denote
$$
\mathbb{E}_{j+1}:=\{u\in H_{0}^{1}(\Omega): \langle u, e_{i}\rangle_{H_{0}^{1}}=0, \forall i=1, 2,...,j\},
$$
while $\mathbb{Y}_{j}:=\mbox{span}\{e_{1},...,e_{j}\}$ will denote the linear subspace generated by the first $j$ eigenfunctions of $-\Delta$ for any $j\in\mathbb{N}$. It is easily seen that $\mathbb{Y}_{j}$ is finite dimensional and $\mathbb{Y}_{j}\oplus\mathbb{E}_{j+1}=H_{0}^{1}(\Omega)$.

In order to study the problem by variational methods, we  introduce the energy functional associated to equation \eqref{CCE} by
$$
J_{\lambda}(u)=\frac{1}{2}\int_{\Omega}|\nabla u|^{2}dx-\frac{1}{22_{\mu}^{\ast}}\int_{\Omega}
\int_{\Omega}\frac{|u(x)|^{2_{\mu}^{\ast}}|u(y)|^{2_{\mu}^{\ast}}}{|x-y|^{\mu}}dxdy-\frac{\lambda}{2}\int_{\Omega}|u|^{2}dx.
$$
Then the Hardy-Littlewood-Sobolev inequality implies $J_{\lambda}$ belongs to $C^{1}(H_{0}^{1}(\Omega),\R)$ with
\begin{equation}
\langle J_{\lambda}^{'}(u),\varphi\rangle=\int_{\Omega}\nabla u\nabla\varphi dx-\int_{\Omega}\int_{\Omega}\frac{|u(x)|^{2_{\mu}^{\ast}}|u(y)|^{2_{\mu}^{\ast}-2}u(y)\varphi(y)}{|x-y|^{\mu}}dxdy-\lambda\int_{\Omega}u\varphi dx
\end{equation}
for all $\varphi\in C_{0}^{\infty}(\Omega)$.  And so  $u$ is a weak solution of \eqref{CCE} if and only if $u$ is a critical point of functional $J_{\lambda}$.

The main results of this paper are stated in the following two theorems.
\begin{thm}\label{EXS}
Assume $\Omega$ is a bounded domain of $\mathbb{R}^N$, with Lipschitz boundary and $0<\mu<N$, the following result holds true:\\
(i) If $N\geq4$, then problem \eqref{CCE} has a nontrivial solution for $\lambda>0$, provided $\lambda$ is not an eigenvalue of $-\Delta$ with homogeneous Dirichlet boundary data.\\
(ii) If $N=3$, then there exist $\lambda_{\ast}$ such that problem \eqref{CCE} has a nontrivial solution for $\lambda>\lambda_{\ast}$, provided $\lambda$ is not an eigenvalue of $-\Delta$ with homogeneous Dirichlet boundary data.
\end{thm}

\begin{thm}\label{NEXS}
If $N\geq3$, $\lambda<0$ and $\Omega\neq \mathbb{R}^N$ is a smooth (possibly unbounded)domain in $\mathbb{R}^N$, which is strictly star-shaped with respect to the origin in $\mathbb{R}^N$, then any solution $u\in H_{0}^{1}(\Omega)$ of problem \eqref{CCE} is trivial.
\end{thm}

\ \ \ \ Throughout this paper we denote the norm $\|u\|:=\left(\int_{\Omega}|\nabla u|^{2}dx\right)^{\frac{1}{2}}$ on $H_{0}^{1}(\Omega)$ and  write $|\cdot|_{q}$ for the $L^{q}(\Omega)$-norm for $q\in[1,\infty]$ and always assume $\Omega$ is a bounded domain of $\mathbb{R}^N$ with Lipschitz boundary, $\lambda$ is a real parameter. We denote positive constants by $C, C_{1}, C_{2}, C_{3} \cdots$.

\begin{Def}  Let $I$  be a  $C^1$ functional defined on Banach space $X$,  we say that  $\{u_{n}\}$ is a Palais-Smale sequence of $I$  at $c$ ($(PS)_c$ sequence, for short) if
\begin{equation}
I(u_{n})\rightarrow c,\ \hbox{and}\ \ \  I^{'}(u_{n})\rightarrow 0,\ \ \  \hbox{as}\ \ n\rightarrow+\infty.
\end{equation}
And we say that $I$ satisfies the  Palais-Smale condition at the level $c$ , if every Palais-Smale sequence  at $c$ has  a convergent subsequence.
\end{Def}

An outline of the paper is as follows: In Section 2, we give some preliminary results and prove $(PS)$ condition. In Section 3, we prove the existence of solutions for \eqref{CCE} when $N\geq4$ and $0<\lambda<\lambda_{1}$ by the Mountain pass theorem. In Section 4,  we prove the existence of solutions for \eqref{CCE} when $N\geq4$ and $\lambda>\lambda_{1}$, provided $\lambda$ is not an eigenvalue of $-\Delta$ with homogeneous Dirichlet boundary data, by the Linking Theorem. In Section 5, we investigate the existence of solutions for $\lambda>0$ when $N=3$. In Section 6, we prove a Poho\u{z}aev identity for \eqref{CCE} and use it to prove the nonexistence of solutions.

\section{Preliminary results}

To prove the  $(PS)$ condition, we need a key lemma which is inspired by the Br\'{e}zis-Lieb convergence lemma (see \cite{BL1}). The proof is analogous to that of Lemma 3.5 in \cite{AC} or Lemma 2.4 in \cite{MS1}, but we exhibit it here for completeness. First, we recall that pointwise convergence of a bounded sequence implies weak convergence (see [\cite{Wi2}, Proposition 5.4.7]).

\begin{lem} Let $N\geq3$, $q\in(1,+\infty)$ and $\{u_{n}\}$ is a bounded sequence in $L^{q}(\mathbb{R}^N)$. If $u_{n}\rightarrow u$ almost everywhere in $\mathbb{R}^N$ as $n\rightarrow\infty$, then $u_{n}\rightharpoonup u$ weakly in $L^{q}(\mathbb{R}^N)$.
\end{lem}

\begin{lem} \label{BLN}Let $N\geq3$ and $0<\mu<N$. If $\{u_{n}\}$ is a bounded sequence in $L^{\frac{2N}{N-2}}(\mathbb{R}^N)$ such that $u_{n}\rightarrow u$ almost everywhere in $\mathbb{R}^N$ as $n\rightarrow\infty$, then the following hold,
$$
\int_{\mathbb{R}^N}(|x|^{-\mu}\ast |u_{n}|^{2_{\mu}^{\ast}})|u_{n}|^{2_{\mu}^{\ast}}dx-\int_{\mathbb{R}^N}(|x|^{-\mu}\ast |u_{n}-u|^{2_{\mu}^{\ast}})|u_{n}-u|^{2_{\mu}^{\ast}}dx\rightarrow\int_{\mathbb{R}^N}(|x|^{-\mu}\ast |u|^{2_{\mu}^{\ast}})|u|^{2_{\mu}^{\ast}}dx
$$
as $n\rightarrow\infty$.
\end{lem}
\begin{proof}
First, similarly to the proof of the Br\'{e}zis-Lieb Lemma \cite{BL1}, we know that
\begin{equation}
|u_{n}-u|^{2_{\mu}^{\ast}}-|u_{n}|^{2_{\mu}^{\ast}}\rightarrow|u|^{2_{\mu}^{\ast}}
\end{equation}
in $L^{\frac{2N}{2N-\mu}}(\mathbb{R}^N)$ as $n\rightarrow\infty$. The Hardy-Littlewood-Sobolev inequality implies that
\begin{equation}
|x|^{-\mu}\ast(|u_{n}-u|^{2_{\mu}^{\ast}}-|u_{n}|^{2_{\mu}^{\ast}})\rightarrow|x|^{-\mu}\ast|u|^{2_{\mu}^{\ast}}
\end{equation}
in $L^{\frac{2N}{\mu}}(\mathbb{R}^N)$ as $n\rightarrow\infty$. On the other hand, we notice that
\begin{equation}\aligned
&\int_{\mathbb{R}^N}\big(|x|^{-\mu}\ast |u_{n}|^{2_{\mu}^{\ast}}\big)|u_{n}|^{2_{\mu}^{\ast}}dx-\int_{\mathbb{R}^N}\big(|x|^{-\mu}\ast |u_{n}-u|^{2_{\mu}^{\ast}}\big)|u_{n}-u|^{2_{\mu}^{\ast}}dx\\
&=\int_{\mathbb{R}^N}\big(|x|^{-\mu}\ast (|u_{n}|^{2_{\mu}^{\ast}}-|u_{n}-u|^{2_{\mu}^{\ast}})\big)(|u_{n}|^{2_{\mu}^{\ast}}-|u_{n}-u|^{2_{\mu}^{\ast}})dx\\
&+2\int_{\mathbb{R}^N}\big(|x|^{-\mu}\ast (|u_{n}|^{2_{\mu}^{\ast}}-|u_{n}-u|^{2_{\mu}^{\ast}})\big)|u_{n}-u|^{2_{\mu}^{\ast}}dx.\\
\endaligned
\end{equation}
By Lemma 2.2, we have that
\begin{equation}
|u_{n}-u|^{2_{\mu}^{\ast}}\rightharpoonup0
\end{equation}
in $L^{\frac{2N}{2N-\mu}}(\mathbb{R}^N)$ as $n\rightarrow\infty$. From (2.2)-(2.5), we know that the result holds.
\end{proof}

\begin{lem}\label{EN}
Assume $N\geq3$ and $0<\mu<N$. Then
$$
\|\cdot\|_{NL}:=\left(\int_{\Omega}\int_{\Omega}\frac{|\cdot|^{2_{\mu}^{\ast}}|\cdot|^{2_{\mu}^{\ast}}}
{|x-y|^{\mu}}dxdy\right)^{\frac{1}{22_{\mu}^{\ast}}}
$$
defines an norm on $L^{2^{\ast}}(\Omega_{1})$.
\end{lem}
\begin{proof} 
By the semigroup property of the Riesz potential,  we obtain
$$
\int_{\Omega}\int_{\Omega}\frac{|u(x)|^{2_{\mu}^{\ast}}|u(y)|^{2_{\mu}^{\ast}}}
{|x-y|^{\mu}}dxdy
=\int_{\Omega}\left(\int_{\Omega}\frac{|u(y)|^{2_{\mu}^{\ast}}}
{|x-y|^{\frac{N+\mu}{2}}}dy\right)^{2}dx
$$
for every $u\in L^{2^{\ast}}(\Omega)$. Then, by the Minkowski inequality, we know, for any $x\in\Omega$
$$\aligned
\left(\int_{\Omega}\frac{|u(y)+v(y)|^{2_{\mu}^{\ast}}}
{|x-y|^{\frac{N+\mu}{2}}}dy\right)^{2}&=\left(\int_{\Omega}\left|\frac{u(y)}
{|x-y|^{\frac{N+\mu}{2}\cdot\frac{1}{2_{\mu}^{\ast}}}}+\frac{v(y)}
{|x-y|^{\frac{N+\mu}{2}\cdot\frac{1}{2_{\mu}^{\ast}}}}\right|^{2_{\mu}^{\ast}}dy
\right)^{\frac{1}{2_{\mu}^{\ast}}\cdot22_{\mu}^{\ast}}\\
&\leq\left(\left(\int_{\Omega}\frac{|u(y)|^{2_{\mu}^{\ast}}}
{|x-y|^{\frac{N+\mu}{2}}}dy\right)^{2\cdot\frac{1}{22_{\mu}^{\ast}}}+\left(\int_{\Omega}\frac{|v(y)|^{2_{\mu}^{\ast}}}
{|x-y|^{\frac{N+\mu}{2}}}dy\right)^{2\cdot\frac{1}{22_{\mu}^{\ast}}}\right)^{22_{\mu}^{\ast}}.
\endaligned
$$
Notice that the integrals are nonnegative and so, by the Minkowski inequality again, we have
$$\aligned
&\left(\int_{\Omega}\left(\int_{\Omega}\frac{|u(y)+v(y)|^{2_{\mu}^{\ast}}}
{|x-y|^{\frac{N+\mu}{2}}}dy\right)^{2}dx\right)^{\frac{1}{22_{\mu}^{\ast}}}\\
&\leq\left(\int_{\Omega}\left(\int_{\Omega}\frac{|u(y)|^{2_{\mu}^{\ast}}}
{|x-y|^{\frac{N+\mu}{2}}}dy\right)^{2}dx\right)^{\frac{1}{22_{\mu}^{\ast}}}
+\left(\int_{\Omega}\left(\int_{\Omega}\frac{|v(y)|^{2_{\mu}^{\ast}}}
{|x-y|^{\frac{N+\mu}{2}}}dy\right)^{2}dx\right)^{\frac{1}{22_{\mu}^{\ast}}},
\endaligned
$$
that is
$$
\|u+v\|_{NL}\leq \|u\|_{NL}+\|v\|_{NL}
$$
for every $u,v\in L^{2^{\ast}}(\Omega)$.
So, it is easy to verify that $\|\cdot\|_{NL}$ is a norm on $L^{2^{\ast}}(\Omega)$.
\end{proof}

\begin{lem} \label{WSo}
Let $N\geq3$, $0<\mu<N$ and $\lambda>0$. If $\{u_{n}\}$ is a  $(PS)_c$ sequence of $J_{\lambda}$, then
$\{u_{n}\}$ is bounded. Let $u_0\in H_{0}^{1}(\Omega)$ be the weak limit of $\{u_{n}\}$, then $u_0$ is a weak solution of problem \eqref{CCE}.
\end{lem}
\begin{proof}  It is easy to see $c\geq0$ and there exists $C_{1}>0$ such that
$$
|J_{\lambda}(u_{n})|\leq C_{1},\ \ \
|\langle J_{\lambda}^{'}(u_{n}),\frac{u_{n}}{\|u_{n}\|}\rangle|\leq C_{1}.
$$
Let $\beta\in(\frac{1}{22_{\mu}^{\ast}},\frac{1}{2})$. For $n$ large enough, we have
$$\aligned
C_{1}(1+\|u_{n}\|)&\geq J_{\lambda}(u_{n})-\beta\langle J_{\lambda}^{'}(u_{n}),u_{n}\rangle\\
&=(\frac{1}{2}-\beta)(\|u_{n}\|^{2}-\lambda|u_{n}|_{2}^{2})+(\beta-\frac{1}{22_{\mu}^{\ast}})
\int_{\Omega}\int_{\Omega}\frac{|u_{n}(x)|^{2_{\mu}^{\ast}}|u_{n}(y)|^{2_{\mu}^{\ast}}}{|x-y|^{\mu}}dxdy\\
&\geq(\frac{1}{2}-\beta)(\delta\|z_{n}\|^{2}+(\lambda_{1}-\lambda)|y_{n}|_{2}^{2})+(\beta-\frac{1}{22_{\mu}^{\ast}})
\int_{\Omega}\int_{\Omega}\frac{|u_{n}(x)|^{2_{\mu}^{\ast}}|u_{n}(y)|^{2_{\mu}^{\ast}}}{|x-y|^{\mu}}dxdy,
\endaligned
$$
where $u_{n}=z_{n}+y_{n}$, $z_{n}\in\mathbb{E}_{j+1}$, $y_{n}\in \mathbb{Y}_{j}$. It is then easy to verify that $\{u_{n}\}$ is bounded in $H_{0}^{1}(\Omega)$ using the fact that that $\mathbb{Y}_{j}$ is finite dimensional and Lemma \ref{EN}.

Since  $H_{0}^{1}(\Omega)$ is  reflexive, up to a subsequence, still denoted by $u_{n}$, there exists $u_0\in H_{0}^{1}(\Omega)$ such that $u_{n}\rightharpoonup u_0$ in $H_{0}^{1}(\Omega)$ and
$
u_{n}\rightharpoonup u_0
$
in $L^{2^{*}}(\Omega) $ as $n\rightarrow+\infty$. Then
$$
|u_{n}|^{2_{\mu}^{*}}\rightharpoonup |u_0|^{2_{\mu}^{*}} \hspace{3.14mm} \mbox{in} \hspace{3.14mm} L^{\frac{2N}{2N-\mu}}(\Omega)
$$
as $n\rightarrow+\infty$. By the Hardy-Littlewood-Sobolev inequality,
the Riesz potential defines a linear continuous map from  $L^{\frac{2N}{2N-\mu}}(\Omega)$ to $L^{\frac{2N}{\mu}}(\Omega)$,  we know that
$$
|x|^{-\mu}\ast|u_{n}|^{2_{\mu}^{*}}\rightharpoonup |x|^{-\mu}\ast|u_0|^{2_{\mu}^{*}} \hspace{3.14mm} \mbox{in} \hspace{3.14mm} L^{\frac{2N}{\mu}}(\Omega)
$$
as $n\rightarrow+\infty$. Combining with the fact that
$$
|u_{n}|^{2_{\mu}^{\ast}-2}u_{n}\rightharpoonup |u_0|^{2_{\mu}^{\ast}-2}u_0 \hspace{3.14mm} \mbox{in} \hspace{3.14mm} L^{\frac{2N}{N-\mu+2}}(\Omega)
$$
as $n\rightarrow+\infty$, we have
$$
(|x|^{-\mu}\ast|u_{n}|^{2_{\mu}^{*}})|u_{n}|^{2_{\mu}^{\ast}-2}u_{n}\rightharpoonup (|x|^{-\mu}\ast|u_0|^{2_{\mu}^{*}})|u_0|^{2_{\mu}^{\ast}-2}u_0 \hspace{3.14mm} \mbox{in} \hspace{3.14mm} L^{\frac{2N}{N+2}}(\Omega)
$$
as $n\rightarrow+\infty$. Since, for any $\varphi\in H_{0}^{1}(\Omega)$,
$$
0\leftarrow\langle J_{\lambda}^{'}(u_{n}),\varphi\rangle=\int_{\Omega}\nabla u_{n}\nabla\varphi dx
-\lambda\int_{\Omega}u_{n}\varphi dx-
\int_{\Omega}\int_{\Omega}\frac{|u_{n}(x)|^{2_{\mu}^{\ast}}|u_{n}(y)|^{2_{\mu}^{\ast}-2}u_{n}(y)\varphi(y)}{|x-y|^{\mu}}dxdy.
$$
Passing to the limit as $n\rightarrow+\infty$ we obtain
$$
\int_{\Omega}\nabla u_0\nabla\varphi dx
-\lambda\int_{\Omega}u_0\varphi dx-
\int_{\Omega}\int_{\Omega}\frac{|u_0(x)|^{2_{\mu}^{\ast}}|u_0(y)|^{2_{\mu}^{\ast}-2}u_0(y)\varphi(y)}{|x-y|^{\mu}}dxdy=0
$$
for any $\varphi\in H_{0}^{1}(\Omega)$, which means $u_0$ is a weak solution of problem \eqref{CCE}.

Finally, taking $\varphi=u_0\in H_{0}^{1}(\Omega)$ as a test function in \eqref{CCE}, we have
$$
\int_{\Omega}|\nabla u_0|^{2}dx=
\lambda\int_{\Omega}u_0^{2}dx+
\int_{\Omega}\int_{\Omega}\frac{|u_0(x)|^{2_{\mu}^{\ast}}|u_0(y)|^{2_{\mu}^{\ast}}}{|x-y|^{\mu}}dxdy,
$$
and so
$$
J_{\lambda}(u_0)=\frac{N+2-\mu}{4N-2\mu}\int_{\Omega}\int_{\Omega}\frac{|u_0(x)|^{2_{\mu}^{\ast}}|u_0(y)|^{2_{\mu}^{\ast}}}
{|x-y|^{\mu}}dxdy\geq0.
$$
\end{proof}

\begin{lem}\label{ConPro} Let $N\geq3$, $0<\mu<N$ and $\lambda>0$. If $\{u_{n}\}$ is a $(PS)_c$ sequence of $J_{\lambda}$ with
\begin{equation}
c<\frac{N+2-\mu}{4N-2\mu}S_{H,L}^{\frac{2N-\mu}{N+2-\mu}},
\end{equation}
then  $\{u_{n}\}$ has a convergent  subsequence.
\end{lem}
\begin{proof}
Let $u_0$ be the weak limit of  $\{u_{n}\}$ obtained in  Lemma \ref{WSo} and define  $v_{n}:=u_{n}-u_0$, then we know $v_{n}\rightharpoonup0$ in $H_{0}^{1}(\Omega)$ and $v_{n}\rightarrow 0$ a.e. in $\Omega$. Moreover, by the Br\'{e}zis-Lieb Lemma  in \cite{BL1} and Lemma \ref{BLN}, we know
$$
\int_{\Omega}|\nabla u_{n}|^{2}dx=\int_{\Omega}|\nabla v_{n}|^{2}dx
+\int_{\Omega}|\nabla u_0|^{2}dx+o(1),
$$
$$
\int_{\Omega}|u_{n}|^{2}dx=\int_{\Omega}|v_{n}|^{2}dx+\int_{\Omega}|u_0|^{2}dx+o_n(1)
$$
and
$$
\int_{\Omega}\int_{\Omega}\frac{|u_{n}(x)|^{2_{\mu}^{\ast}}|u_{n}(y)|^{2_{\mu}^{\ast}}}
{|x-y|^{\mu}}dxdy=\int_{\Omega}\int_{\Omega}\frac{|v_{n}(x)|^{2_{\mu}^{\ast}}|v_{n}(y)|^{2_{\mu}^{\ast}}}
{|x-y|^{\mu}}dxdy+\int_{\Omega}\int_{\Omega}\frac{|u_0(x)|^{2_{\mu}^{\ast}}|u_0(y)|^{2_{\mu}^{\ast}}}
{|x-y|^{\mu}}dxdy+o_n(1).
$$
Then, we have
\begin{equation}\label{C1}
\aligned
c\leftarrow J_{\lambda}(u_{n})&=\frac{1}{2}\int_{\Omega}|\nabla u_{n}|^{2}dx-\frac{\lambda}{2}\int_{\Omega} u_{n}^{2}dx-\frac{1}{22_{\mu}^{\ast}}\int_{\Omega}\int_{\Omega}\frac{|u_{n}(x)|^{2_{\mu}^{\ast}}|u_{n}(y)|^{2_{\mu}^{\ast}}}
{|x-y|^{\mu}}dxdy\\
&=\frac{1}{2}\int_{\Omega}|\nabla v_{n}|^{2}dx-\frac{\lambda}{2}\int_{\Omega}v_{n}^{2}dx+\frac{1}{2}\int_{\Omega}|\nabla u_0|^{2}dx-\frac{\lambda}{2}\int_{\Omega} u_0^{2}dx\\
&\hspace{0.5cm}-\frac{1}{22_{\mu}^{\ast}}\int_{\Omega}\int_{\Omega}\frac{|v_{n}(x)|^{2_{\mu}^{\ast}}|v_{n}(y)|^{2_{\mu}^{\ast}}}
{|x-y|^{\mu}}dxdy-\frac{1}{22_{\mu}^{\ast}}\int_{\Omega}\int_{\Omega}\frac{|u_0(x)|^{2_{\mu}^{\ast}}|u_0(y)|^{2_{\mu}^{\ast}}}
{|x-y|^{\mu}}dxdy+o_n(1)\\
&=J_{\lambda}(u_0)+\frac{1}{2}\int_{\Omega}|\nabla v_{n}|^{2}dx-\frac{\lambda}{2}\int_{\Omega} v_{n}^{2}dx-\frac{1}{22_{\mu}^{\ast}}\int_{\Omega}\int_{\Omega}\frac{|v_{n}(x)|^{2_{\mu}^{\ast}}|v_{n}(y)|^{2_{\mu}^{\ast}}}
{|x-y|^{\mu}}dxdy+o_n(1)\\
&\geq \frac{1}{2}\int_{\Omega}|\nabla v_{n}|^{2}dx-\frac{1}{22_{\mu}^{\ast}}\int_{\Omega}\int_{\Omega}\frac{|v_{n}(x)|^{2_{\mu}^{\ast}}|v_{n}(y)|^{2_{\mu}^{\ast}}}
{|x-y|^{\mu}}dxdy+o_n(1),
\endaligned
\end{equation}
since $J_{\lambda}(u_0)\geq 0$ and $\displaystyle\int_{\Omega}v_{n}^{2}\rightarrow 0$, as $n\rightarrow+\infty$.
Similarly, since  $\langle J_{\lambda}^{'}(u_0),u_0\rangle=0$, we have

\begin{equation}\label{C2}
\aligned
o_n(1)&= \langle J_{\lambda}^{'}(u_{n}),u_{n}\rangle\\
& =\int_{\Omega}|\nabla u_{n}|^{2}dx-\lambda\int_{\Omega} u_{n}^{2}dx-\int_{\Omega}\int_{\Omega}\frac{|u_{n}(x)|^{2_{\mu}^{\ast}}|u_{n}(y)|^{2_{\mu}^{\ast}}}
{|x-y|^{\mu}}dxdy\\
&=\int_{\Omega}|\nabla v_{n}|^{2}dx-\lambda\int_{\Omega} v_{n}^{2}dx+\int_{\Omega}|\nabla u_0|^{2}dx-\lambda\int_{\Omega} u_0^{2}dx\\
&\hspace{0.5cm}-\int_{\Omega}\int_{\Omega}\frac{|v_{n}(x)|^{2_{\mu}^{\ast}}|v_{n}(y)|^{2_{\mu}^{\ast}}}
{|x-y|^{\mu}}dxdy-\int_{\Omega}\int_{\Omega}\frac{|u_0(x)|^{2_{\mu}^{\ast}}|u_0(y)|^{2_{\mu}^{\ast}}}
{|x-y|^{\mu}}dxdy+o_n(1)\\
&=\langle J_{\lambda}^{'}(u_0),u_0\rangle+\int_{\Omega}|\nabla v_{n}|^{2}dx-\lambda\int_{\Omega}v_{n}^{2}dx-\int_{\Omega}\int_{\Omega}\frac{|v_{n}(x)|^{2_{\mu}^{\ast}}|v_{n}(y)|^{2_{\mu}^{\ast}}}
{|x-y|^{\mu}}dxdy+o_n(1)\\
&=\int_{\Omega}|\nabla v_{n}|^{2}dx-\int_{\Omega}\int_{\Omega}\frac{|v_{n}(x)|^{2_{\mu}^{\ast}}|v_{n}(y)|^{2_{\mu}^{\ast}}}
{|x-y|^{\mu}}dxdy+o_n(1).
\endaligned
\end{equation}
From \eqref{C2},  we know there exists a nonnegative constant $b$ such that
$$
\int_{\Omega}|\nabla v_{n}|^{2}dx\rightarrow b
$$
and
$$
\int_{\Omega}\int_{\Omega}\frac{|v_{n}(x)|^{2_{\mu}^{\ast}}|v_{n}(y)|^{2_{\mu}^{\ast}}}
{|x-y|^{\mu}}dxdy\rightarrow b,
$$
as $n\rightarrow+\infty$.
From \eqref{C1} and  \eqref{C2}, we obtain
\begin{equation}\label{C3}
c\geq \frac{N+2-\mu}{4N-2\mu}b.
\end{equation}
By the definition of the best constant $S_{H,L}$ in \eqref{S1}, we have
$$
S_{H,L}\left(\int_{\Omega}\int_{\Omega}\frac{|v_{n}(x)|^{2_{\mu}^{\ast}}|v_{n}(y)|^{2_{\mu}^{\ast}}}
{|x-y|^{\mu}}dxdy\right)^{\frac{N-2}{2N-\mu}}\leq\int_{\Omega}|\nabla v_{n}|^{2}dx,
$$
which yields $b\geq S_{H,L}b^{\frac{N-2}{2N-\mu}}$. Thus we have either $b=0$ or $b\geq S_{H,L}^{\frac{2N-\mu}{N-\mu+2}}$. If $b=0$, the proof is complete. Otherwise $b\geq S_{H,L}^{\frac{2N-\mu}{N-\mu+2}}$,  then we obtain from \eqref{C3},
$$
\frac{N+2-\mu}{4N-2\mu}S_{H,L}^{\frac{2N-\mu}{N-\mu+2}}\leq\frac{N+2-\mu}{4N-2\mu}b\leq c,
$$
which contradicts with the fact that $c<\frac{N+2-\mu}{4N-2\mu}S_{H,L}^{\frac{2N-\mu}{N+2-\mu}}$. Thus $b=0$, and
$$
\|u_{n}-u_0\|\rightarrow0
$$
as $n\rightarrow+\infty$. This ends the proof of Lemma \ref{ConPro}.
\end{proof}

\section{The case $N\geq4$, $0<\lambda<\lambda_{1}$}
 We devote this Section to prove Theorem \ref{EXS} for the case $N\geq4$ and $0<\lambda<\lambda_{1}$.

By Lemma \ref{ExFu}, we know that
$U(x)=\frac{[N(N-2)]^{\frac{N-2}{4}}}{(1+|x|^{2})^{\frac{N-2}{2}}}$  is a minimizer for both $S$ and $S_{H,L}$.
Without loss of generality, we may assume that $0\in\Omega$ and $B_{\delta}\subset\Omega\subset B_{2\delta}$. Let $\psi\in C_{0}^{\infty}(\Omega)$ such that
$$
\left\{\begin{array}{l}
\displaystyle \psi(x)=\left\{\begin{array}{l}
\displaystyle 1 \hspace{13.14mm} \mbox{if}\hspace{2.14mm} x\in B_{\delta},\\
\displaystyle 0 \hspace{13.14mm} \mbox{if} \hspace{2.14mm}x\in \mathbb{R}^N \setminus\Omega,\\
\end{array}
\right.\\
\displaystyle 0\leq\psi(x)\leq1 \hspace{19.14mm} \forall x\in \mathbb{R}^N ,\\
\displaystyle |D\psi(x)|\leq C=const \hspace{13.14mm} \forall x\in \mathbb{R}^N .\\
\end{array}
\right.
$$
We define, for $\varepsilon>0$,
$$\aligned
U_{\varepsilon}(x)&:=\varepsilon^{\frac{2-N}{2}}U(\frac{x}{\varepsilon}),\\
u_{\varepsilon}(x)&:=\psi(x)U_{\varepsilon}(x).\\
\endaligned
$$
From Lemma 1.46 of \cite{Wi} and Lemma \ref{ExFu}, we know
\begin{equation}\label{E2}
|\nabla U_{\varepsilon}|_{2}^{2}=|U_{\varepsilon}|_{2^{\ast}}^{2^{\ast}}=S^{\frac{N}{2}},
\end{equation}
and as $\varepsilon\rightarrow0^{+}$,
\begin{equation}\label{E5}
\int_{\Omega}|\nabla u_{\varepsilon}|^{2}dx=S^{\frac{N}{2}}+O(\varepsilon^{N-2})
=C(N,\mu))^{\frac{N-2}{2N-\mu}\cdot\frac{N}{2}}S_{H,L}^{\frac{N}{2}}+O(\varepsilon^{N-2}),
\end{equation}
\begin{equation}\label{E3}
\int_{\Omega}|u_{\varepsilon}|^{2^{\ast}}dx=S^{\frac{N}{2}}+O(\varepsilon^{N})
\end{equation}
and
\begin{equation}\label{E4}
\int_{\Omega}|u_{\varepsilon}|^{2}dx\geq\left\{\begin{array}{l}
\displaystyle d\varepsilon^{2}|ln\varepsilon|+O(\varepsilon^{2}) \hspace{10.14mm} \mbox{if}\hspace{2.14mm} N=4,\\
\displaystyle d\varepsilon^{2}+O(\varepsilon^{N-2}) \hspace{13.14mm} \mbox{if} \hspace{2.14mm}N\geq5,\\
\end{array}
\right.
\end{equation}
where $d$ is a positive constant. 

Using the Hardy-Littlewood-Sobolev inequality, on one hand, we get
\begin{equation}\label{E6}
\aligned
\left(\int_{\Omega}\int_{\Omega}\frac{|u_{\varepsilon}(x)|^{2_{\mu}^{\ast}}|u_{\varepsilon}(y)|^{2_{\mu}^{\ast}}}
{|x-y|^{\mu}}dxdy\right)^{\frac{N-2}{2N-\mu}}
&\leq C(N,\mu)^{\frac{N-2}{2N-\mu}}|u_{\varepsilon}|_{2^{\ast}}^{2}\\
&=C(N,\mu)^{\frac{N-2}{2N-\mu}}\big(S^{\frac{N}{2}}+O(\varepsilon^{N})\big)^{\frac{N-2}{N}}\\
&=C(N,\mu)^{\frac{N-2}{2N-\mu}}\big(C(N,\mu)^{\frac{N-2}{2N-\mu}\cdot\frac{N}{2}}S_{H,L}^{\frac{N}{2}}+O(\varepsilon^{N})\big)^{\frac{N-2}{N}}\\
&=C(N,\mu)^{\frac{N-2}{2N-\mu}\cdot\frac{N}{2}}S_{H,L}^{\frac{N-2}{2}}+O(\varepsilon^{N-2}).
\endaligned
\end{equation}
While on the other hand,
\begin{equation}\label{E7}
\aligned
\int_{\Omega}\int_{\Omega}\frac{|u_{\varepsilon}(x)|^{2_{\mu}^{\ast}}|u_{\varepsilon}(y)|^{2_{\mu}^{\ast}}}
{|x-y|^{\mu}}dxdy
&\geq \int_{B_{\delta}}\int_{B_{\delta}}\frac{|u_{\varepsilon}(x)|^{2_{\mu}^{\ast}}|u_{\varepsilon}(y)|^{2_{\mu}^{\ast}}}{|x-y|^{\mu}}dxdy\\
&=\int_{B_{\delta}}\int_{B_{\delta}}\frac{|U_{\varepsilon}(x)|^{2_{\mu}^{\ast}}|U_{\varepsilon}(y)|^{2_{\mu}^{\ast}}}{|x-y|^{\mu}}dxdy\\
&=\int_{\mathbb{R}^N}\int_{\mathbb{R}^N}\frac{|U_{\varepsilon}(x)|^{2_{\mu}^{\ast}}|U_{\varepsilon}(y)|^{2_{\mu}^{\ast}}}{|x-y|^{\mu}}dxdy\\
&\hspace{1cm}-2\int_{\mathbb{R}^N\setminus B_{\delta}}\int_{B_{\delta}}\frac{|U_{\varepsilon}(x)|^{2_{\mu}^{\ast}}|U_{\varepsilon}(y)|^{2_{\mu}^{\ast}}}{|x-y|^{\mu}}dxdy\\
&\hspace{1.5cm}-\int_{\mathbb{R}^N\setminus B_{\delta}}\int_{\mathbb{R}^N\setminus B_{\delta}}\frac{|U_{\varepsilon}(x)|^{2_{\mu}^{\ast}}|U_{\varepsilon}(y)|^{2_{\mu}^{\ast}}}{|x-y|^{\mu}}dxdy\\
&=C(N,\mu)^{\frac{N}{2}}S_{H,L}^{\frac{2N-\mu}{2}}-2\D-\E,
\endaligned
\end{equation}
where
$$
\D=\int_{\mathbb{R}^N\setminus B_{\delta}}\int_{B_{\delta}}\frac{|U_{\varepsilon}(x)|^{2_{\mu}^{\ast}}|U_{\varepsilon}(y)|^{2_{\mu}^{\ast}}}{|x-y|^{\mu}}dxdy\ $$
and
$$
\E=\int_{\mathbb{R}^N\setminus B_{\delta}}\int_{\mathbb{R}^N\setminus B_{\delta}}\frac{|U_{\varepsilon}(x)|^{2_{\mu}^{\ast}}|U_{\varepsilon}(y)|^{2_{\mu}^{\ast}}}{|x-y|^{\mu}}dxdy.
$$

 By direct computation, we know
\begin{equation}\label{E8}
\aligned
\D&=\int_{\mathbb{R}^N\setminus B_{\delta}}\int_{B_{\delta}}\frac{|U_{\varepsilon}(x)|^{2_{\mu}^{\ast}}|U_{\varepsilon}(y)|^{2_{\mu}^{\ast}}}{|x-y|^{\mu}}dxdy\\
&=\int_{\mathbb{R}^N\setminus B_{\delta}}\int_{B_{\delta}}\frac{\varepsilon^{\mu-2N}[N(N-2)]^{\frac{2N-\mu}{2}}}
{(1+|\frac{x}{\varepsilon}|^{2})^{\frac{2N-\mu}{2}}|x-y|^{\mu}(1+|\frac{y}{\varepsilon}|^{2})^{\frac{2N-\mu}{2}}}dxdy\\
&=\varepsilon^{2N-\mu}[N(N-2)]^{\frac{2N-\mu}{2}}\int_{\mathbb{R}^N\setminus B_{\delta}}\int_{B_{\delta}}\frac{1}
{(\varepsilon^{2}+|x|^{2})^{\frac{2N-\mu}{2}}|x-y|^{\mu}(\varepsilon^{2}+|y|^{2})^{\frac{2N-\mu}{2}}}dxdy\\
&\leq O(\varepsilon^{2N-\mu})\left(\int_{\mathbb{R}^N\setminus B_{\delta}}\frac{1}
{(\varepsilon^{2}+|x|^{2})^{N}}dx\right)^{\frac{2N-\mu}{2N}}\left(\int_{B_{\delta}}\frac{1}{(\varepsilon^{2}+|y|^{2})^{N}}dy\right)^{\frac{2N-\mu}{2N}}\\
&\leq O(\varepsilon^{2N-\mu})\left(\int_{\mathbb{R}^N\setminus B_{\delta}}\frac{1}
{|x|^{2N}}dx\right)^{\frac{2N-\mu}{2N}}\left(\int_{0}^{\delta}\frac{r^{N-1}}{(\varepsilon^{2}+r^{2})^{N}}dr\right)^{\frac{2N-\mu}{2N}}\\
&=O(\varepsilon^{\frac{2N-\mu}{2}})\left(\int_{0}^{\frac{\delta}{\varepsilon}}\frac{z^{N-1}}{(1+z^{2})^{N}}dz\right)^{\frac{2N-\mu}{2N}}\\
&\leq O(\varepsilon^{\frac{2N-\mu}{2}})\left(\int_{0}^{+\infty}\frac{z^{N-1}}{(1+z^{2})^{N}}dz\right)^{\frac{2N-\mu}{2N}}\\
&=O(\varepsilon^{\frac{2N-\mu}{2}})
\endaligned
\end{equation}
and
\begin{equation}\label{E9}
\aligned
\E&=\int_{\mathbb{R}^N\setminus B_{\delta}}\int_{\mathbb{R}^N\setminus B_{\delta}}\frac{|U_{\varepsilon}(x)|^{2_{\mu}^{\ast}}|U_{\varepsilon}(y)|^{2_{\mu}^{\ast}}}{|x-y|^{\mu}}dxdy\\
&=\int_{\mathbb{R}^N\setminus B_{\delta}}\int_{\mathbb{R}^N\setminus B_{\delta}}\frac{\varepsilon^{\mu-2N}[N(N-2)]^{\frac{2N-\mu}{2}}}
{(1+|\frac{x}{\varepsilon}|^{2})^{\frac{2N-\mu}{2}}|x-y|^{\mu}(1+|\frac{y}{\varepsilon}|^{2})^{\frac{2N-\mu}{2}}}dxdy\\
&=\varepsilon^{2N-\mu}[N(N-2)]^{\frac{2N-\mu}{2}}\int_{\mathbb{R}^N\setminus B_{\delta}}\int_{\mathbb{R}^N\setminus B_{\delta}}\frac{1}
{(\varepsilon^{2}+|x|^{2})^{\frac{2N-\mu}{2}}|x-y|^{\mu}(\varepsilon^{2}+|y|^{2})^{\frac{2N-\mu}{2}}}dxdy\\
&\leq\varepsilon^{2N-\mu}[N(N-2)]^{\frac{2N-\mu}{2}}\int_{\mathbb{R}^N\setminus B_{\delta}}\int_{\mathbb{R}^N\setminus B_{\delta}}\frac{1}
{|x|^{2N-\mu}|x-y|^{\mu}|y|^{2N-\mu}}dxdy\\
&=O(\varepsilon^{2N-\mu}).\\
\endaligned
\end{equation}
It follows from \eqref{E7} to \eqref{E9}  that
\begin{equation}\label{E10}
\aligned
\left(\int_{\Omega}\int_{\Omega}\frac{|u_{\varepsilon}(x)|^{2_{\mu}^{\ast}}|u_{\varepsilon}(y)|^{2_{\mu}^{\ast}}}
{|x-y|^{\mu}}dxdy\right)^{\frac{N-2}{2N-\mu}}
&\geq \left(C(N,\mu)^{\frac{N}{2}}S_{H,L}^{\frac{2N-\mu}{2}}-O(\varepsilon^{\frac{2N-\mu}{2}})-O(\varepsilon^{2N-\mu})\right)^{\frac{N-2}{2N-\mu}}\\
&=\left(C(N,\mu)^{\frac{N}{2}}S_{H,L}^{\frac{2N-\mu}{2}}-O(\varepsilon^{\frac{2N-\mu}{2}})\right)^{\frac{N-2}{2N-\mu}}.\\
\endaligned
\end{equation}

When $N=3$, \eqref{E5} and \eqref{E10} also hold.

\begin{lem}\label{Element} If $N\geq4$ and $\lambda>0$, then,
there exists $v\in H_{0}^{1}(\Omega)\backslash\{{0}\}$ such that
$$
\frac{|\nabla v|_{2}^{2}-\lambda|v|_{2}^{2}}
{\|v\|_{NL}^{2}}<S_{H,L}.
$$
\end{lem}
\begin{proof}
If $N=4$, from \eqref{E4}, \eqref{E5} and \eqref{E10}, we can obtain
\begin{equation}\label{E11}\aligned
\frac{|\nabla u_{\varepsilon}|_{2}^{2}-\lambda|u_{\varepsilon}|_{2}^{2}}
{\|u_{\varepsilon}\|_{NL}^{2}}
&\leq \frac{C(4,\mu)^{\frac{4}{8-\mu}}S_{H,L}^{2}-\lambda d\varepsilon^{2}|ln\varepsilon| +O(\varepsilon^{2})}{\left(C(4,\mu)^{2}S_{H,L}^{\frac{8-\mu}{2}}-O(\varepsilon^{4-\frac{\mu}{2}})\right)^{\frac{2}{8-\mu}}}\\
&=S_{H,L}-\frac{\lambda d\varepsilon^{2}|ln\varepsilon|}{\left(C(4,\mu)^{2}S_{H,L}^{\frac{8-\mu}{2}}-O(\varepsilon^{4-\frac{\mu}{2}})\right)^{\frac{2}{8-\mu}}}+O(\varepsilon^{2})\\
&\leq S_{H,L}-\lambda d\varepsilon^{2}|ln\varepsilon|+O(\varepsilon^{2})\\
&<S_{H,L}.
\endaligned
\end{equation}
If $N\geq5$, using  \eqref{E4}, \eqref{E5} and \eqref{E10} again, we have
\begin{equation}\aligned
\frac{|\nabla u_{\varepsilon}|_{2}^{2}-\lambda|u_{\varepsilon}|_{2}^{2}}
{\|u_{\varepsilon}\|_{NL}^{2}}
&\leq \frac{C(N,\mu)^{\frac{N-2}{2N-\mu}\cdot\frac{N}{2}}S_{H,L}^{\frac{N}{2}}-\lambda d\varepsilon^{2}+O(\varepsilon^{N-2})}
{\left(C(N,\mu)^{\frac{N}{2}}S_{H,L}^{\frac{2N-\mu}{2}}-O(\varepsilon^{N-\frac{\mu}{2}})\right)^{\frac{N-2}{2N-\mu}}}\\
&\leq S_{H,L}-\frac{\lambda d\varepsilon^{2}}{\left(C(N,\mu)^{\frac{N}{2}}S_{H,L}^{\frac{2N-\mu}{2}}
-O(\varepsilon^{N-\frac{\mu}{2}})\right)^{\frac{N-2}{2N-\mu}}}+O(\varepsilon^{\frac{N}{2}})\\
&\leq S_{H,L}-\lambda d\varepsilon^{2}+O(\varepsilon^{\frac{N}{2}})\\
&<S_{H,L}.
\endaligned
\end{equation}
From the arguments above, we may take $v:=u_{\varepsilon}$ with $\vr$ small enough and then the conclusion follows immediatelly.
\end{proof}

\begin{lem} \label{MPG}If $N\geq3$ and $\lambda\in(0,\lambda_{1})$, then, the functional $J_{\lambda}$ satisfies the following properties:\\
(i) There exist $\alpha,\rho>0$ such that $J_{\lambda}(u)\geq\alpha$ for $\|u\|=\rho$. \\
(ii) There exists $e\in H_{0}^{1}(\Omega)$ with $\|e\|>\rho$ such that $J_{\lambda}(e)<0$.
\end{lem}
\begin{proof} (i) By $\lambda\in(0,\lambda_{1})$, the Sobolev embedding and the Hardy-Littlewood-Sobolev inequality, for all $u\in H_{0}^{1}(\Omega)\backslash\ \{0\}$ we have
$$
\aligned
J_{\lambda}(u)&\geq\frac{1}{2}\int_{\Omega}|\nabla u|^{2}dx-\frac{\lambda}{2\lambda_{1}}\int_{\Omega}|\nabla u|^{2}dx-\frac{1}{22_{\mu}^{\ast}}C_{0}|u|_{2^{\ast}}^{2(\frac{2N-\mu}{N-2})}\\
&\geq\frac{1}{2}(1-\frac{\lambda}{\lambda_{1}})\|u\|^{2}-\frac{1}{22_{\mu}^{\ast}}C_{0}C_{1}\|u\|^{2(\frac{2N-\mu}{N-2})}.\\
\endaligned
$$
Since $2<2(\frac{2N-\mu}{N-2})$, we can choose some $\alpha,\rho>0$ such that $J_{\lambda}(u)\geq\alpha$ for $\|u\|=\rho$.

(ii) For some $u_{1}\in H_{0}^{1}(\Omega)\backslash\ \{0\}$, we have
$$
J_{\lambda}(tu_{1})=\frac{t^{2}}{2}\int_{\Omega}|\nabla u_{1}|^{2}dx-\frac{\lambda t^{2}}{2}\int_{\Omega} u_{1}^{2}dx-\frac{t^{22_{\mu}^{\ast}}}{22_{\mu}^{\ast}}
\int_{\Omega}\int_{\Omega}\frac{|u_{1}(x)|^{2_{\mu}^{\ast}}|u_{1}(y)|^{2_{\mu}^{\ast}}}
{|x-y|^{\mu}}dxdy<0
$$
for $t>0$ large enough. Hence, we can take an $e:=t_{1}u_{1}$ for some $t_{1}>0$ and (ii) follows.
\end{proof}

\begin{Prop}\label{PSS}
By Lemma \ref{MPG} and the mountain pass theorem without $(PS)$ condition (cf. \cite{Wi}), there
exists a $(PS)$ sequence $\{u_{n}\}$ such that $J_{\lambda}(u_{n})\rightarrow c$ and $ J_{\lambda}^{'}(u_{n})\rightarrow0$ in $H_{0}^{1}(\Omega)^{-1}$ at the minimax level
\begin{equation}\label{MPL}
c^*=\inf\limits_{\gamma\in\Gamma}\max\limits_{t\in[0,1]}J_{\lambda}(\gamma(t))>0,
\end{equation}
where
$$
\Gamma:=\{\gamma\in C([0,1],H_{0}^{1}(\Omega)):\gamma(0)=0,J_{\lambda}(\gamma(1))<0\}.
$$
\end{Prop}
\noindent
{\bf Proof of Theorem \ref{EXS}: Case  $N\geq4$, $0<\lambda<\lambda_{1}$.}
 From Lemma \ref{Element}, we know there exists $v\in H_{0}^{1}(\Omega)\backslash\{{0}\}$ such that
$$
\frac{|\nabla v|_{2}^{2}-\lambda|v|_{2}^{2}}
{\|v\|_{NL}^{2}}<S_{H,L}.
$$
Therefore,
$$
\aligned
0<\max_{t\geq0}J_{\lambda}(tv)&=\max_{t\geq0}\left\{\frac{t^{2}}{2}\int_{\Omega}|\nabla v|^{2}dx-\frac{\lambda t^{2}}{2}\int_{\Omega} v^{2}dx-\frac{t^{22_{\mu}^{\ast}}}{22_{\mu}^{\ast}}
\int_{\Omega}\int_{\Omega}\frac{|v(x)|^{2_{\mu}^{\ast}}|v(y)|^{2_{\mu}^{\ast}}}
{|x-y|^{\mu}}dxdy\right\}\\
&=\frac{N+2-\mu}{4N-2\mu}\left(\frac{|\nabla v|_{2}^{2}-\lambda|v|_{2}^{2}}
{\|v\|_{NL}^{2}}\right)^{\frac{2N-\mu}{N+2-\mu}}\\
&<\frac{N+2-\mu}{4N-2\mu}S_{H,L}^{\frac{2N-\mu}{N+2-\mu}}.\\
\endaligned
$$
By the definition of $c^*$, we know $c^*<\frac{N+2-\mu}{4N-2\mu}S_{H,L}^{\frac{2N-\mu}{N+2-\mu}}$. Let $\{u_{n}\}$ be the  $(PS)$ sequence obtained in Proposition \ref{PSS}. Applying Lemma \ref{ConPro}, we know  $\{u_{n}\}$ contains a convergent subsequence. And so, we have $J_{\lambda}$ has a critical value $c^*\in\big(0, \frac{N+2-\mu}{4N-2\mu}S_{H,L}^{\frac{2N-\mu}{N+2-\mu}}\big)$ and the problem \eqref{CCE} has a nontrivial solution.$\hfill{} \Box$

\section{The case $N\geq4$, $\lambda\geq\lambda_{1}$}
\ \ \ \ We may suppose that $\lambda\in[\lambda_{j}, \lambda_{j+1})$ for some $j\in\mathbb{N}$, where $\lambda_{j}$ is the $j$-th eigenvalue of $-\Delta$ on $\Omega$ with boundary condition $u=0$. $e_{j}$ is the $j$-th eigenfunctions corresponding to the eigenvalue $\lambda_{j}$.

\begin{lem}\label{LK} If $N\geq3$ and $\lambda\in[\lambda_{j}, \lambda_{j+1})$ for some $j\in\mathbb{N}$, then, the functional $J_{\lambda}$ satisfies the following properties:\\
(i) There exist $\alpha,\rho>0$ such that for any $u\in\mathbb{E}_{j+1}$ with $\|u\|=\rho$ it results that $J_{\lambda}(u)\geq\alpha$. \\
(ii) $J_{\lambda}(u)<0$ for any $u\in \mathbb{Y}_{j}$.\\
(iii) Let $\mathbb{F}$ be a finite dimensional subspace of $H_{0}^{1}(\Omega)$. There exists $R>\rho$ such that for any $u\in\mathbb{F}$ with $\|u\|\geq R$ it results that $J_{\lambda}(u)\leq0$.
\end{lem}
\begin{proof} (i) Since $\lambda\in[\lambda_{j}, \lambda_{j+1})$, by  the Sobolev embedding and the Hardy-Littlewood-Sobolev inequality, for all $u\in\mathbb{E}_{j+1}\backslash\ \{0\}$ we have
$$
\aligned
J_{\lambda}(u)&\geq\frac{1}{2}\int_{\Omega}|\nabla u|^{2}dx-\frac{\lambda}{2\lambda_{j+1}}\int_{\Omega}|\nabla u|^{2}dx-\frac{1}{22_{\mu}^{\ast}}C_{0}|u|_{2^{\ast}}^{2(\frac{2N-\mu}{N-2})}\\
&\geq\frac{1}{2}(1-\frac{\lambda}{\lambda_{j+1}})\|u\|^{2}-\frac{1}{22_{\mu}^{\ast}}C_{1}\|u\|^{2(\frac{2N-\mu}{N-2})}.\\
\endaligned
$$
Since $2<2(\frac{2N-\mu}{N-2})$, we can choose some $\alpha,\rho>0$ such that $J_{\lambda}(u)\geq\alpha$ for $u\in\mathbb{E}_{j+1}$ with $\|u\|=\rho$.

(ii) Let $u\in \mathbb{Y}_{j}$, that is, $u=\sum_{i=1}^{j}l_{i}e_{i},$
where $l_{i}\in \mathbb{R},i=1,...,j.$ Since $\{e_{i}\}_{i\in\mathbb{N}}$ is an orthonormal basis of $L^{2}(\Omega)$ and $H_{0}^{1}(\Omega)$, we have
$$
\int_{\Omega}u^{2}dx=\sum_{i=1}^{j}l_{i}^{2}\ \ \ \ \mbox{and}\ \ \ \int_{\Omega}|\nabla u|^{2}dx=\sum_{i=1}^{j}l_{i}^{2}|\nabla e_{i}|_{2}^{2}.
$$
Then, we get
$$
\aligned
J_{\lambda}(u)&=\frac{1}{2}\sum_{i=1}^{j}l_{i}^{2}(|\nabla e_{i}|_{2}^{2}-\lambda)-\frac{1}{22_{\mu}^{\ast}}\int_{\Omega}\int_{\Omega}\frac{|u(x)|^{2_{\mu}^{\ast}}|u(y)|^{2_{\mu}^{\ast}}}
{|x-y|^{\mu}}dxdy\\
&< \frac{1}{2}\sum_{i=1}^{j}l_{i}^{2}(\lambda_{i}-\lambda)\\
&\leq 0,\\
\endaligned
$$
thanks to $\lambda_{i}\leq\lambda_{j}\leq\lambda$.

(iii) For $u\in \mathbb{F}\backslash\ \{0\}$, by the non-negativity of $\lambda$ gives
$$\aligned
J_{\lambda}(u)&=\frac{1}{2}\|u\|^{2}-\frac{\lambda }{2}|u|_{2}^{2}-\frac{1}{22_{\mu}^{\ast}}
\|u\|_{NL}^{22_{\mu}^{\ast}}\\
&\leq\frac{1}{2}\|u\|^{2}-\frac{1}{22_{\mu}^{\ast}}\|u\|_{NL}^{22_{\mu}^{\ast}}\\
&\leq\frac{1}{2}\|u\|^{2}-\frac{C_{1}}{22_{\mu}^{\ast}}\|u\|^{22_{\mu}^{\ast}}
\endaligned
$$
for some positive constant $C_{1}>0$, since all norms on finite dimensional space are equivalent. So,
$
J_{\lambda}(u)\rightarrow-\infty
$
as $\|u\|\rightarrow+\infty$. Hence, there exists $R>\rho$ such that for any $u\in\mathbb{F}$ with $\|u\|\geq R$ it results that $J_{\lambda}(u)\leq0$ and (iii) follows.
\end{proof}

From Lemma \ref{Element}, if $N\geq4$ and $\lambda>0$, then for $\vr$ small enough,
$$
\frac{|\nabla u_{\varepsilon}|_{2}^{2}-\lambda|u_{\varepsilon}|_{2}^{2}}
{\|u_{\varepsilon}\|_{NL}^{2}}<S_{H,L}.
$$
For any $j\in\mathbb{N}$, we define the linear space
$$
\mathbb{G}_{j,\varepsilon}:=\mbox{span}\{e_{1},...,e_{j}, u_{\varepsilon}\}
$$
and set
$$
m_{j,\varepsilon}:=\max_{u\in\mathbb{G}_{j,\varepsilon}, \|u\|_{NL}=1}\left(\int_{\Omega}|\nabla u|^{2}dx-\lambda\int_{\Omega}|u|^{2}dx\right),
$$
where $\|\cdot\|_{NL}$ is defined in Lemma \ref{EN}.

\begin{lem} \label{LE1} If $N\geq4$ and $\lambda\in[\lambda_{j}, \lambda_{j+1})$ for some $j\in\mathbb{N}$, then,\\
(i) $m_{j,\varepsilon}$ is achieved at some $u_{m}\in \mathbb{G}_{j,\varepsilon}$ and $u_{m}$ can be written as follows
$$
u_{m}=v+tu_{\varepsilon}
$$
with $v\in\mathbb{Y}_{j}$ and $t\geq0$. \\
(ii) The following estimate holds true
\begin{equation}\label{MAX}
m_{j,\varepsilon}\leq\left\{\begin{array}{l}
\displaystyle (\lambda_{j}-\lambda)|v|_{2}^{2} \hspace{10.14mm} \mbox{if}\hspace{2.14mm} t=0,\\
\displaystyle (\lambda_{j}-\lambda)|v|_{2}^{2}+A_{\varepsilon}\left(1+|v|_{2}O(\varepsilon^{\frac{N-2}{2}})\right)
+O(\varepsilon^{\frac{N-2}{2}})|v|_{2} \hspace{3.14mm} \mbox{if} \hspace{2.14mm}t>0,\\
\end{array}
\right.
\end{equation}
as $\varepsilon\rightarrow0$, where $v$ is given in (i), $u_{\varepsilon}$ is given in Section 3 and
\begin{equation}\label{AE}
A_{\varepsilon}=\frac{|\nabla u_{\varepsilon}|_{2}^{2}-\lambda|u_{\varepsilon}|_{2}^{2}}
{\|u_{\varepsilon}\|_{NL}^{2}}.
\end{equation}
\end{lem}
\begin{proof}
(i) Since $\mathbb{G}_{j,\varepsilon}$ is a finite dimensional space, then $m_{\varepsilon}$ is achieved at some $u_{m}\in\mathbb{G}_{j,\varepsilon}$, that is,
$$
m_{j,\varepsilon}=|\nabla u_{m}|_{2}^{2}-\lambda|u_{m}|_{2}^{2}
\ \ \ \mbox{and}\ \ \ \|u_{m}\|_{NL}=1.
$$
Obviously, $u_{m}\not\equiv0$. From the definition of $\mathbb{G}_{j,\varepsilon}$ we have that
$$
u_{m}=v+tu_{\varepsilon}
$$
for some $v\in\mathbb{Y}_{j}$ and $t\in  \R$. We can suppose that $t\geq0$, otherwise, if $t<0$ we can replace $u_{m}$ with $-u_{m}$. The result follows.

(ii)
If $t=0$, then $u_{m}=v\in\mathbb{Y}_{j}$ and
$$
m_{j,\varepsilon}=|\nabla u_{m}|_{2}^{2}-\lambda|u_{m}|_{2}^{2}=|\nabla v|_{2}^{2}-\lambda|v|_{2}^{2}\leq (\lambda_{j}-\lambda)|v|_{2}^{2}.
$$

We consider the case $t>0$. Since $e_{1},...,e_{j}\in L^{\infty}(\Omega)$, we also have $v\in L^{\infty}(\Omega)$. By a direct computation, we have
$$
\aligned
\int_{B_{2\delta}}\int_{B_{2\delta}}&\frac{|u_{\varepsilon}(x)|^{2_{\mu}^{\ast}}
|u_{\varepsilon}(y)|^{2_{\mu}^{\ast}-1}}
{|x-y|^{\mu}}dxdy\\
&=\int_{B_{2\delta}}\int_{B_{2\delta}}\frac{|U_{\varepsilon}(x)|^{2_{\mu}^{\ast}}
|U_{\varepsilon}(y)|^{2_{\mu}^{\ast}-1}}{|x-y|^{\mu}}dxdy\\
&=\varepsilon^{\frac{2\mu-3N-2}{2}}[N(N-2)]^{\frac{3N-2\mu+2}{4}}
\int_{B_{2\delta}}\int_{B_{2\delta}}\frac{1}
{(1+|\frac{x}{\varepsilon}|^{2})^{\frac{2N-\mu}{2}}|x-y|^{\mu}(1+|\frac{y}
{\varepsilon}|^{2})^{\frac{N-\mu+2}{2}}}dxdy\\
&=\varepsilon^{\frac{2\mu-3N-2}{2}}[N(N-2)]^{\frac{3N-2\mu+2}{4}}\varepsilon^{2N-\mu}
\int_{B_{\frac{2\delta}{\varepsilon}}}\int_{B_{\frac{2\delta}{\varepsilon}}}\frac{1}
{(1+|x|^{2})^{\frac{2N-\mu}{2}}|x-y|^{\mu}
(1+|y|^{2})^{\frac{N-\mu+2}{2}}}dxdy\\
&\leq O(\varepsilon^{\frac{N-2}{2}})\int_{\R^{N}}\int_{\R^{N}}\frac{1}
{(1+|x|^{2})^{\frac{2N-\mu}{2}}|x-y|^{\mu}
(1+|y|^{2})^{\frac{N-\mu+2}{2}}}dxdy.\\
\endaligned
$$
If $\mu>1$, by Hardy-Littlewood-Sobolev inequality, we have
$$
\aligned
\int_{B_{2\delta}}\int_{B_{2\delta}}&\frac{|u_{\varepsilon}(x)|^{2_{\mu}^{\ast}}
|u_{\varepsilon}(y)|^{2_{\mu}^{\ast}-1}}
{|x-y|^{\mu}}dxdy\\
&\leq O(\varepsilon^{\frac{N-2}{2}})\left(\int_{\R^{N}}\left(\frac{1}
{(1+|x|^{2})^{\frac{2N-\mu}{2}}}\right)^{\frac{N}{N-1}}dx\right)^{\frac{N-1}{N}}
\left(\int_{\R^{N}}\left(\frac{1}
{(1+|x|^{2})^{\frac{N-\mu+2}{2}}}\right)^{\frac{N}{N-\mu+1}}dx\right)^{\frac{N-\mu+1}{N}}\\
&=O(\varepsilon^{\frac{N-2}{2}}).
\endaligned
$$
If $\mu\leq1$, by Hardy-Littlewood-Sobolev inequality again, we have
$$
\aligned
\int_{B_{2\delta}}\int_{B_{2\delta}}&\frac{|u_{\varepsilon}(x)|^{2_{\mu}^{\ast}}
|u_{\varepsilon}(y)|^{2_{\mu}^{\ast}-1}}
{|x-y|^{\mu}}dxdy\\
&\leq O(\varepsilon^{\frac{N-2}{2}})\left(\int_{\R^{N}}\left(\frac{1}
{(1+|x|^{2})^{\frac{2N-\mu}{2}}}\right)^{\frac{2N}{2N-\mu}}dx\right)^{\frac{2N-\mu}{2N}}
\left(\int_{\R^{N}}\left(\frac{1}
{(1+|x|^{2})^{\frac{N-\mu+2}{2}}}\right)^{\frac{2N}{2N-\mu}}dx\right)^{\frac{2N-\mu}{2N}}\\
&=O(\varepsilon^{\frac{N-2}{2}}).
\endaligned
$$
Thus, we obtain
$$
\int_{\Omega}\int_{\Omega}\frac{|u_{\varepsilon}(x)|^{2_{\mu}^{\ast}}
|u_{\varepsilon}(y)|^{2_{\mu}^{\ast}-1}}
{|x-y|^{\mu}}dxdy\leq O(\varepsilon^{\frac{N-2}{2}}).
$$
On the other hand, by a direct computation, we have
$$
\aligned
\int_{B_{\delta}}\int_{B_{\delta}}&\frac{|u_{\varepsilon}(x)|^{2_{\mu}^{\ast}}
|u_{\varepsilon}(y)|^{2_{\mu}^{\ast}-1}}
{|x-y|^{\mu}}dxdy\\
&=\int_{B_{\delta}}\int_{B_{\delta}}\frac{|U_{\varepsilon}(x)|^{2_{\mu}^{\ast}}
|U_{\varepsilon}(y)|^{2_{\mu}^{\ast}-1}}{|x-y|^{\mu}}dxdy\\
&=\varepsilon^{\frac{2\mu-3N-2}{2}}[N(N-2)]^{\frac{3N-2\mu+2}{4}}
\int_{B_{\delta}}\int_{B_{\delta}}\frac{1}
{(1+|\frac{x}{\varepsilon}|^{2})^{\frac{2N-\mu}{2}}|x-y|^{\mu}(1+|\frac{y}
{\varepsilon}|^{2})^{\frac{N-\mu+2}{2}}}dxdy\\
&=\varepsilon^{\frac{2\mu-3N-2}{2}}[N(N-2)]^{\frac{3N-2\mu+2}{4}}\varepsilon^{2N-\mu}
\int_{B_{\frac{\delta}{\varepsilon}}}\int_{B_{\frac{\delta}{\varepsilon}}}\frac{1}
{(1+|x|^{2})^{\frac{2N-\mu}{2}}|x-y|^{\mu}
(1+|y|^{2})^{\frac{N-\mu+2}{2}}}dxdy\\
&\geq O(\varepsilon^{\frac{N-2}{2}})\int_{B_{\delta}}\int_{B_{\delta}}\frac{1}
{(1+|x|^{2})^{\frac{2N-\mu}{2}}|x-y|^{\mu}
(1+|y|^{2})^{\frac{N-\mu+2}{2}}}dxdy\\
&=O(\varepsilon^{\frac{N-2}{2}})
\endaligned
$$
provided $\varepsilon<1$ and so
$$
\int_{\Omega}\int_{\Omega}\frac{|u_{\varepsilon}(x)|^{2_{\mu}^{\ast}}
|u_{\varepsilon}(y)|^{2_{\mu}^{\ast}-1}}
{|x-y|^{\mu}}dxdy\geq O(\varepsilon^{\frac{N-2}{2}}).
$$
Then we can get
$$
\int_{\Omega}\int_{\Omega}\frac{|u_{\varepsilon}(x)|^{2_{\mu}^{\ast}}
|u_{\varepsilon}(y)|^{2_{\mu}^{\ast}-1}}
{|x-y|^{\mu}}dxdy= O(\varepsilon^{\frac{N-2}{2}}).
$$

By convexity, we obtain
\begin{equation}\label{ESLC2}
\aligned
1&=\int_{\Omega}\int_{\Omega}\frac{|u_{m}(x)|^{2_{\mu}^{\ast}}|u_{m}(y)|^{2_{\mu}^{\ast}}}
{|x-y|^{\mu}}dxdy\\
&=\int_{\Omega}\int_{\Omega}\frac{|v(x)+tu_{\varepsilon}(x)|^{2_{\mu}^{\ast}}|v(y)+tu_{\varepsilon}(y)|^{2_{\mu}^{\ast}}}
{|x-y|^{\mu}}dxdy\\
&\geq\int_{\Omega}\int_{\Omega}\frac{|tu_{\varepsilon}(x)|^{2_{\mu}^{\ast}}|tu_{\varepsilon}(y)|^{2_{\mu}^{\ast}}}
{|x-y|^{\mu}}dxdy+22_{\mu}^{\ast}\int_{\Omega}\int_{\Omega}\frac{|tu_{\varepsilon}(x)|^{2_{\mu}^{\ast}-1}v(x)|tu_{\varepsilon}(y)|^{2_{\mu}^{\ast}}}
{|x-y|^{\mu}}dxdy\\
&\hspace{0.5cm}+{2_{\mu}^{\ast}}^{2}\int_{\Omega}\int_{\Omega}\frac{|tu_{\varepsilon}(x)|^{2_{\mu}^{\ast}-1}v(x)|tu_{\varepsilon}(y)|^{2_{\mu}^{\ast}-1}v(y)}
{|x-y|^{\mu}}dxdy\\
&\geq\int_{\Omega}\int_{\Omega}\frac{|tu_{\varepsilon}(x)|^{2_{\mu}^{\ast}}|tu_{\varepsilon}(y)|^{2_{\mu}^{\ast}}}
{|x-y|^{\mu}}dxdy+22_{\mu}^{\ast}\int_{\Omega}\int_{\Omega}\frac{|tu_{\varepsilon}(x)|^{2_{\mu}^{\ast}-1}v(x)|tu_{\varepsilon}(y)|^{2_{\mu}^{\ast}}}
{|x-y|^{\mu}}dxdy\\
&\geq t^{22_{\mu}^{\ast}}\int_{\Omega}\int_{\Omega}\frac{|u_{\varepsilon}(x)|^{2_{\mu}^{\ast}}|u_{\varepsilon}(y)|^{2_{\mu}^{\ast}}}
{|x-y|^{\mu}}dxdy-22_{\mu}^{\ast}t^{22_{\mu}^{\ast}-1}|v|_{\infty}\int_{\Omega}\int_{\Omega}\frac{|u_{\varepsilon}(x)|^{2_{\mu}^{\ast}-1}
|u_{\varepsilon}(y)|^{2_{\mu}^{\ast}}}
{|x-y|^{\mu}}dxdy\\
&\geq t^{22_{\mu}^{\ast}}\int_{\Omega}\int_{\Omega}\frac{|u_{\varepsilon}(x)|^{2_{\mu}^{\ast}}|u_{\varepsilon}(y)|^{2_{\mu}^{\ast}}}
{|x-y|^{\mu}}dxdy-C_{2}t^{22_{\mu}^{\ast}-1}|v|_{2}O(\varepsilon^{\frac{N-2}{2}}),\\
\endaligned
\end{equation}
where we used the fact that  $\mathbb{Y}_{j}$ is a finite dimensional space and all norms on  $\mathbb{Y}_{j}$ are equivalent. This implies that $t<C_{3}$  for some constant $C_{3}>0$. Taking \eqref{ESLC2} into account, we have
$$
\int_{\Omega}\int_{\Omega}\frac{|tu_{\varepsilon}(x)|^{2_{\mu}^{\ast}}|tu_{\varepsilon}(y)|^{2_{\mu}^{\ast}}}
{|x-y|^{\mu}}dxdy\leq 1+O(\varepsilon^{\frac{N-2}{2}})|v|_{2}.
$$
By \eqref{AE}, one can see that
$$\aligned
m_{j,\varepsilon}&=\int_{\Omega}|\nabla (v+tu_{\varepsilon})|^{2}dx-\lambda\int_{\Omega}|v+tu_{\varepsilon}|^{2}dx\\
&\leq (\lambda_{j}-\lambda)|v|_{2}^{2}+A_{\varepsilon}\left(\int_{\Omega}\int_{\Omega}\frac{|tu_{\varepsilon}(x)|^{2_{\mu}^{\ast}}|tu_{\varepsilon}(y)|^{2_{\mu}^{\ast}}}
{|x-y|^{\mu}}dxdy\right)^{\frac{N-2}{2N-\mu}}+C_{4}|u_{\varepsilon}|_{1}|v|_{2}\\
&\leq (\lambda_{j}-\lambda)|v|_{2}^{2}+A_{\varepsilon}\left(1+|v|_{2}O(\varepsilon^{\frac{N-2}{2}})\right)^{\frac{N-2}{2N-\mu}}
+C_{4}|u_{\varepsilon}|_{1}|v|_{2}\\
&\leq (\lambda_{j}-\lambda)|v|_{2}^{2}+A_{\varepsilon}\left(1+|v|_{2}O(\varepsilon^{\frac{N-2}{2}})\right)
+O(\varepsilon^{\frac{N-2}{2}})|v|_{2},
\endaligned
$$
where we had used the estimate in Lemma 2.25 of \cite{Wi} that
$
|u_{\varepsilon}|_{1}=O(\varepsilon^{\frac{N-2}{2}}).
$
\end{proof}

\begin{lem} \label{LE}
If $N\geq4$ and $\lambda\in(\lambda_{j}, \lambda_{j+1})$ for some $j\in\mathbb{N}$, then,
$$
\frac{|\nabla u|_{2}^{2}-\lambda|u|_{2}^{2}}
{\|u\|_{NL}^{2}}<S_{H,L}
$$
for any $u\in\mathbb{G}_{j,\varepsilon}$.
\end{lem}
\begin{proof}
We only need to check that
$$
m_{j,\varepsilon}=\max_{u\in\mathbb{G}_{j,\varepsilon}, \|u\|_{NL}=1}\left(\int_{\Omega}|\nabla u|^{2}dx-\lambda\int_{\Omega}|u|^{2}dx\right)<S_{H,L}.
$$
 If $t=0$ in \eqref{MAX}, by the choice of $\lambda\in(\lambda_{j}, \lambda_{j+1})$, we get that
$$
m_{j,\varepsilon}\leq (\lambda_{j}-\lambda)|v|_{2}^{2}<0<S_{H,L}.
$$
Now we suppose that $t>0$ and discuss the cases $N\geq5$ and $N=4$ separately.

If $N\geq5$, we have
$$\aligned
m_{j,\varepsilon}&\leq(\lambda_{j}-\lambda)|v|_{2}^{2}+\frac{|\nabla u_{\varepsilon}|_{2}^{2}-\lambda|u_{\varepsilon}|_{2}^{2}}
{\|u_{\varepsilon}\|_{NL}^{2}}\left(1+|v|_{2}O(\varepsilon^{\frac{N-2}{2}})\right)
+O(\varepsilon^{\frac{N-2}{2}})|v|_{2}\\
&\leq(\lambda_{j}-\lambda)|v|_{2}^{2}+\frac{C(N,\mu)^{\frac{N-2}{2N-\mu}\cdot\frac{N}{2}}S_{H,L}^{\frac{N}{2}}-\lambda d\varepsilon^{2}+O(\varepsilon^{N-2})}
{\left(C(N,\mu)^{\frac{N}{2}}S_{H,L}^{\frac{2N-\mu}{2}}-O(\varepsilon^{N-\frac{\mu}{2}})\right)^{\frac{N-2}{2N-\mu}}}\left(1+|v|_{2}O(\varepsilon^{\frac{N-2}{2}})\right)
+O(\varepsilon^{\frac{N-2}{2}})|v|_{2}\\
&\leq\left(S_{H,L}-\frac{\lambda d\varepsilon^{2}}{\left(C(N,\mu)^{\frac{N}{2}}S_{H,L}^{\frac{2N-\mu}{2}}-O(\varepsilon^{N-\frac{\mu}{2}})\right)^{\frac{N-2}{2N-\mu}}}+O(\varepsilon^{\frac{N}{2}})\right)
\left(1+|v|_{2}O(\varepsilon^{\frac{N-2}{2}})\right)+(\lambda_{j}-\lambda)|v|_{2}^{2}+O(\varepsilon^{\frac{N-2}{2}})|v|_{2}\\
&\leq S_{H,L}-\frac{\lambda d\varepsilon^{2}}{\left(C(N,\mu)^{\frac{N}{2}}S_{H,L}^{\frac{2N-\mu}{2}}
-O(\varepsilon^{N-\frac{\mu}{2}})\right)^{\frac{N-2}{2N-\mu}}}+O(\varepsilon^{\frac{N}{2}})+(\lambda_{j}-\lambda)|v|_{2}^{2}+O(\varepsilon^{\frac{N-2}{2}})|v|_{2}
\endaligned
$$
for $\varepsilon>0$ sufficiently small. Since $\lambda\in(\lambda_{j}, \lambda_{j+1})$, we know
\begin{equation}\label{EAK}
(\lambda_{j}-\lambda)|v|_{2}^{2}+O(\varepsilon^{\frac{N-2}{2}})|v|_{2}
\leq\frac{1}{4(\lambda_{j}-\lambda)}O(\varepsilon^{N-2})=O(\varepsilon^{N-2}),
\end{equation}
therefore
$$
m_{j,\varepsilon}\leq S_{H,L}-\lambda d\varepsilon^{2}+O(\varepsilon^{\frac{N}{2}})<S_{H,L}
$$
for $\varepsilon>0$ sufficiently small.

If $N=4$,  by \eqref{EAK}, we have
$$\aligned
m_{j,\varepsilon}&\leq(\lambda_{j}-\lambda)|v|_{2}^{2}+\frac{|\nabla u_{\varepsilon}|_{2}^{2}-\lambda|u_{\varepsilon}|_{2}^{2}}
{\|u_{\varepsilon}\|_{NL}^{2}}\left(1+|v|_{2}O(\varepsilon)\right)
+O(\varepsilon)|v|_{2}\\
&\leq(\lambda_{j}-\lambda)|v|_{2}^{2}+\frac{C(4,\mu)^{\frac{4}{8-\mu}}S_{H,L}^{2}-\lambda d\varepsilon^{2}|ln\varepsilon| +O(\varepsilon^{2})}{\left(C(4,\mu)^{2}S_{H,L}^{\frac{8-\mu}{2}}-O(\varepsilon^{4-\frac{\mu}{2}})\right)^{\frac{2}{8-\mu}}}\left(1+|v|_{2}O(\varepsilon)\right)
+O(\varepsilon)|v|_{2}\\
&\leq\left(S_{H,L}-\frac{\lambda d\varepsilon^{2}|ln\varepsilon|}{\left(C(4,\mu)^{2}S_{H,L}^{\frac{8-\mu}{2}}-O(\varepsilon^{4-\frac{\mu}{2}})\right)^{\frac{2}{8-\mu}}}
+O(\varepsilon^{2})\right)\left(1+|v|_{2}O(\varepsilon)\right)+(\lambda_{j}-\lambda)|v|_{2}^{2}+O(\varepsilon)|v|_{2}\\
&\leq S_{H,L}-\frac{\lambda d\varepsilon^{2}|ln\varepsilon|}{\left(C(4,\mu)^{2}S_{H,L}^{\frac{8-\mu}{2}}-O(\varepsilon^{4-\frac{\mu}{2}})\right)^{\frac{2}{8-\mu}}}
+O(\varepsilon^{2})+(\lambda_{j}-\lambda)|v|_{2}^{2}+O(\varepsilon)|v|_{2}\\
&\leq S_{H,L}-\lambda d\varepsilon^{2}|ln\varepsilon|+O(\varepsilon^{2})\\
&<S_{H,L}
\endaligned
$$
for $\varepsilon>0$ sufficiently small. The result follows.
\end{proof}

\textbf{Proof of Theorem \ref{EXS}  $N\geq4$, $\lambda>\lambda_{1}$.} From the definition of $\mathbb{G}_{j,\varepsilon}$ we know
$$
u_{m}=\overline{v}+tz_{\varepsilon},
$$
where
$$
\overline{v}=v+t\sum_{i=1}^{j}\left(\int_{\Omega}u_{\varepsilon}e_{i}dx\right)e_{i}\in\mathbb{Y}_{j}
$$
and
$$
z_{\varepsilon}=u_{\varepsilon}-\sum_{i=1}^{j}\left(\int_{\Omega}u_{\varepsilon}e_{i}dx\right)e_{i},
$$
so that $\overline{v}$ and $z_{\varepsilon}$ are orthogonal in $L^{2}(\Omega)$. This imply that
$$
|u_{m}|_{2}^{2}=|\overline{v}|_{2}^{2}+t^{2}|z_{\varepsilon}|_{2}^{2}.
$$
Then,
$$
\mathbb{G}_{j,\varepsilon}=\mathbb{Y}_{j}\oplus \mathbb{R}z_{\varepsilon}.
$$
Applying Lemma \ref{LK}, we know that $J_{\lambda}$ satisfies the geometric structure of the Linking Theorem (see [\cite{Rph1}, Theorem 5.3]), that is
$$
\inf_{u\in\mathbb{E}_{j+1},\|u\|=\rho}J_{\lambda}(u)\geq\alpha>0,
$$
$$
\sup_{u\in\mathbb{Y}_{j}}J_{\lambda}(u)<0
$$
and
$$
\sup_{u\in\mathbb{G}_{j,\varepsilon},\|u\|\geq R}J_{\lambda}(u)\leq0.
$$
where $\alpha$ and $R$ are as in Lemma \ref{LK}. Define the Linking critical level of $J_{\lambda}$, i.e.
\begin{equation}\label{LL}
c^{\star}=\inf\limits_{\gamma\in\Gamma}\max\limits_{u\in V}J_{\lambda}(\gamma(u))>0,
\end{equation}
where
$$
\Gamma:=\{\gamma\in C(\overline{V},H_{0}^{1}(\Omega)):\gamma=id \hspace{1.6mm}\mbox{on} \hspace{1.6mm}\partial V\}
$$
and
$$
V:=(\overline{B_{R}}\cap \mathbb{Y}_{j})\oplus\{rz_{\varepsilon}:r\in(0,R)\}.
$$
For any $\gamma\in\Gamma$, we have
$$
c^{\star}\leq\max\limits_{u\in V}J_{\lambda}(\gamma(u))
$$
and in particular, if we take $\gamma=id$ on $\overline{V}$, then
$$
c^{\star}\leq\max\limits_{u\in V}J_{\lambda}(u)\leq\max\limits_{u\in \mathbb{G}_{j,\varepsilon}}J_{\lambda}(u).
$$
Note that for any $u\in H_{0}^{1}(\Omega)\backslash\{0\}$,
$$
\max\limits_{t\geq0}J_{\lambda}(tu)=\frac{N+2-\mu}{4N-2\mu}\left(\frac{|\nabla u|_{2}^{2}-\lambda|u|_{2}^{2}}
{\|u\|_{NL}^{2}}\right)^{\frac{2N-\mu}{N+2-\mu}}.
$$
From $\mathbb{G}_{j,\varepsilon}$ is a linear space we have
$$
\max\limits_{u\in \mathbb{G}_{j,\varepsilon}}J_{\lambda}(u)=\max\limits_{u\in \mathbb{G}_{j,\varepsilon},t\neq0}J_{\lambda}(|t|\frac{u}{|t|})=\max\limits_{u\in \mathbb{G}_{j,\varepsilon},t>0}J_{\lambda}(tu)\leq\max\limits_{u\in \mathbb{G}_{j,\varepsilon},t\geq0}J_{\lambda}(tu).
$$
Thus, by Lemma \ref{LE}, we have
$$\aligned
c^{\star}&\leq\max\limits_{u\in \mathbb{G}_{j,\varepsilon},t\geq0}J_{\lambda}(tu)\\
&=\max\limits_{u\in \mathbb{G}_{j,\varepsilon}}\frac{N+2-\mu}{4N-2\mu}\left(\frac{|\nabla u|_{2}^{2}-\lambda|u|_{2}^{2}}
{\|u\|_{NL}^{2}}\right)^{\frac{2N-\mu}{N+2-\mu}}\\
&<\frac{N+2-\mu}{4N-2\mu}S_{H,L}^{\frac{2N-\mu}{N+2-\mu}}.
\endaligned
$$
Therefore, the Linking Theorem and Lemma \ref{ConPro}  yield that problem \eqref{CCE} admits a nontrivial
solution $u\in H_{0}^{1}(\Omega)$ with critical value $c^{\star}\geq\alpha$. $\hfill{} \Box$

\section{The case $N=3$}
In this Section, we prove Theorem \ref{EXS}  for the case $N=3$ by using the Mountain Pass Theorem and the Linking Theorem. We still denote
$\mathbb{F}$ be a finite dimensional subspace of $H_{0}^{1}(\Omega)$ and
$$
\mathbb{G}_{j,\varepsilon}:=\mbox{span}\{e_{1},...,e_{j}, u_{\varepsilon}\}.
$$
for any $j\in\mathbb{N}$.

\begin{lem} \label {3MPLE}Let $N=3$ and $u_{\varepsilon}$ be as in Section 3. Then, there exists $\lambda_{\ast}$ such that for any $\lambda>\lambda_{\ast}$,
$$
\frac{|\nabla u_{\varepsilon}|_{2}^{2}-\lambda|u_{\varepsilon}|_{2}^{2}}
{\|u_{\varepsilon}\|_{NL}^{2}}<S_{H,L}
$$
provided $\varepsilon>0$ is sufficiently small.
\end{lem}
\begin{proof} By the definition of $u_{\varepsilon}$, we can get
\begin{equation}\label{3E1}
\aligned
\int_{\Omega}|u_{\varepsilon}|^{2}dx&
\geq \int_{B_{\delta}}|U_{\varepsilon}|^{2}dx\geq C_{0}\varepsilon\\
\endaligned
\end{equation}
for $\varepsilon>0$ sufficiently small.
By \eqref{E5}, \eqref{E10} and \eqref{3E1}, we have

$$\aligned
\frac{|\nabla u_{\varepsilon}|_{2}^{2}-\lambda|u_{\varepsilon}|_{2}^{2}}
{\|u_{\varepsilon}\|_{NL}^{2}}
&\leq\frac{C(3,\mu)^{\frac{1}{6-\mu}\cdot\frac{3}{2}}S_{H,L}^{\frac{3}{2}}-\lambda C_{0}\varepsilon+O(\varepsilon)}
{\left(C(3,\mu)^{\frac{3}{2}}S_{H,L}^{\frac{6-\mu}{2}}-O(\varepsilon^{3-\frac{\mu}{2}})\right)^{\frac{1}{6-\mu}}
}\\
&=S_{H,L}-\frac{(\lambda C_{0}-O(1))\varepsilon}{\left(C(3,\mu)^{\frac{3}{2}}S_{H,L}^{\frac{6-\mu}{2}}-O(\varepsilon^{3-\frac{\mu}{2}})\right)^{\frac{1}{6-\mu}}}\\
&<S_{H,L}
\endaligned
$$
if $\lambda$ is large enough, say $\lambda>\lambda_{\ast}>0$, while $\varepsilon>0$ is sufficiently small.
\end{proof}

We will show that $J_{\lambda}$ has the geometric structure of the Mountain Pass Theorem when $\lambda\in(0,\lambda_{1})$ and the geometric structure of the Linking Theorem when $\lambda\in[\lambda_{j}, \lambda_{j+1})$ for some $j\in\mathbb{N}$.

We set
$$
m_{j,\varepsilon}:=\max_{u\in\mathbb{G}_{j,\varepsilon}, \|u\|_{NL}=1}\left(\int_{\Omega}|\nabla u|^{2}dx-\lambda\int_{\Omega}|u|^{2}dx\right).
$$
Related to Lemma \ref{LE1}, we also have the corresponding result for  $N=3$, so, we have
\begin{lem}\label{LE2} If $N=3$ and $\lambda\in[\lambda_{j}, \lambda_{j+1})$ for some $j\in\mathbb{N}$, then,\\
(i) $m_{\varepsilon}$ is achieved in $u_{m}\in \mathbb{G}_{j,\varepsilon}$ and $u_{m}$ can be written as follows
$$
u_{m}=v+tu_{\varepsilon}
$$
with $v\in\mathbb{Y}_{j}$ and $t\geq0$. \\
(ii) The following estimate holds true
\begin{equation}\label{3EFL}
m_{j,\varepsilon}\leq\left\{\begin{array}{l}
\displaystyle (\lambda_{j}-\lambda)|v|_{2}^{2} \hspace{10.14mm} \mbox{if}\hspace{2.14mm} t=0,\\
\displaystyle (\lambda_{j}-\lambda)|v|_{2}^{2}+A_{\varepsilon}\left(1+|v|_{2}O(\varepsilon^{\frac{1}{2}})\right)
+O(\varepsilon^{\frac{1}{2}})|v|_{2} \hspace{3.14mm} \mbox{if} \hspace{2.14mm}t>0,\\
\end{array}
\right.
\end{equation}
as $\varepsilon\rightarrow0$, where $v$ is given in (i), $u_{\varepsilon}$ is given in Section 3 and $$A_{\varepsilon}=\frac{|\nabla u_{\varepsilon}|_{2}^{2}-\lambda|u_{\varepsilon}|_{2}^{2}}
{\|u_{\varepsilon}\|_{NL}^{2}}.$$
\end{lem}

\begin{lem} \label{3EL}If $N=3$, $\lambda\in(\lambda_{j}, \lambda_{j+1})$ for some $j\in\mathbb{N}$ and $\lambda>\lambda_{\ast}$, then,
$$
\frac{|\nabla u|_{2}^{2}-\lambda|u|_{2}^{2}}
{\|u\|_{NL}^{2}}<S_{H,L}
$$
for any $u\in\mathbb{G}_{j,\varepsilon}$.
\end{lem}
\begin{proof} If $t=0$ in \eqref{3EFL}, by the choice of $\lambda\in(\lambda_{j}, \lambda_{j+1})$, we get that
$$
m_{\varepsilon}\leq (\lambda_{j}-\lambda)|v|_{2}^{2}\leq0<S_{H,L}.
$$
When $t>0$, by \eqref{E5}, \eqref{E10},  \eqref{3E1} and Lemma \ref{LE2}, using similar estimate as in \eqref{EAK},  we have
$$\aligned
m_{j,\varepsilon}&\leq(\lambda_{j}-\lambda)|v|_{2}^{2}+\frac{|\nabla u_{\varepsilon}|_{2}^{2}-\lambda|u_{\varepsilon}|_{2}^{2}}
{\|u_{\varepsilon}\|_{NL}^{2}}\left(1+|v|_{2}O(\varepsilon^{\frac{1}{2}})\right)
+O(\varepsilon^{\frac{1}{2}})|v|_{2}\\
&\leq(\lambda_{j}-\lambda)|v|_{2}^{2}+\frac{C(3,\mu)^{\frac{1}{6-\mu}\cdot\frac{3}{2}}S_{H,L}^{\frac{3}{2}}-\lambda C_{0}\varepsilon+O(\varepsilon)}
{\left(C(3,\mu)^{\frac{3}{2}}S_{H,L}^{\frac{6-\mu}{2}}-O(\varepsilon^{3-\frac{\mu}{2}})\right)^{\frac{1}{6-\mu}}
}\left(1+|v|_{2}O(\varepsilon^{\frac{1}{2}})\right)
+O(\varepsilon^{\frac{1}{2}})|v|_{2}\\
&\leq \left(S_{H,L}-\frac{(\lambda C_{0}-O(1))\varepsilon}{\left(C(3,\mu)^{\frac{3}{2}}S_{H,L}^{\frac{6-\mu}{2}}-O(\varepsilon^{3-\frac{\mu}{2}})\right)^{\frac{1}{6-\mu}}}\right)
\left(1+|v|_{2}O(\varepsilon^{\frac{1}{2}})\right)+(\lambda_{j}-\lambda)|v|_{2}^{2}+O(\varepsilon^{\frac{1}{2}})|v|_{2}\\
&\leq S_{H,L}-\frac{(\lambda C_{0}-O(1))\varepsilon}{\left(C(3,\mu)^{\frac{3}{2}}S_{H,L}^{\frac{6-\mu}{2}}-O(\varepsilon^{3-\frac{\mu}{2}})\right)^{\frac{1}{6-\mu}}}
+(\lambda_{j}-\lambda)|v|_{2}^{2}+O(\varepsilon^{\frac{1}{2}})|v|_{2}\\
&\leq S_{H,L}-\lambda C_{0} \varepsilon+O(\varepsilon)\\
&<S_{H,L}
\endaligned
$$
for $\varepsilon>0$ sufficiently small, since $\lambda>\lambda_{\ast}$ and $\lambda\in(\lambda_{j}, \lambda_{j+1})$.
The result follows.
\end{proof}

\textbf{Proof of Theorem \ref{EXS}: Case $N=3$.} We consider the  two
cases: $\lambda_{1}>\lambda_{\ast}$ and $\lambda_{1}>\lambda_{\ast}$ separately.

\textbf{Case 1. } $\lambda_{1}>\lambda_{\ast}$.

For this case we will use the Mountain Pass Theorem if $\lambda\in(\lambda_{\ast},\lambda_{1})$ while the Linking Theorem if $\lambda\in(\lambda_{j}, \lambda_{j+1})$ for some $j\in\mathbb{N}$.

If $\lambda\in(\lambda_{\ast},\lambda_{1})$, by Lemma \ref{MPG} and the mountain pass theorem without $(PS)$ condition (cf. \cite{Wi}), there
exists a $(PS)$ sequence $\{u_{n}\}$ such that $J_{\lambda}(u_{n})\rightarrow c^*$ and $ J_{\lambda}^{'}(u_{n})\rightarrow0$ in $H_{0}^{1}(\Omega)^{-1}$ at the Mountain Pass level $c^*$.
From Lemma \ref{3MPLE}, we have there exists $v\in H_{0}^{1}(\Omega)\backslash\{{0}\}$ such that
$$
\frac{|\nabla v|_{2}^{2}-\lambda|v|_{2}^{2}}
{\|v\|_{NL}^{2}}<S_{H,L}.
$$
Thus,
$$
\aligned
0<\max_{t\geq0}J_{\lambda}(tv)&=\max_{t\geq0}\left\{\frac{t^{2}}{2}\int_{\Omega}|\nabla v|^{2}dx-\frac{\lambda t^{2}}{2}\int_{\Omega} v^{2}dx-\frac{t^{22_{\mu}^{\ast}}}{22_{\mu}^{\ast}}
\int_{\Omega}\int_{\Omega}\frac{|v(x)|^{2_{\mu}^{\ast}}|v(y)|^{2_{\mu}^{\ast}}}
{|x-y|^{\mu}}dxdy\right\}\\
&=\frac{5-\mu}{12-2\mu}\left(\frac{|\nabla v|_{2}^{2}-\lambda|v|_{2}^{2}}
{\|v\|_{NL}^{2}}\right)^{\frac{6-\mu}{5-\mu}}\\
&<\frac{5-\mu}{12-2\mu}S_{H,L}^{\frac{6-\mu}{5-\mu}}.\\
\endaligned
$$
By the definition of $c$, we know $c<\frac{5-\mu}{12-2\mu}S_{H,L}^{\frac{6-\mu}{5-\mu}}$.

From Lemma \ref{ConPro}, we obtain $\{u_{n}\}$ contains a convergent subsequence. So, we have $J_{\lambda}$ has a critical value $c^*\in(0, \frac{5-\mu}{12-2\mu}S_{H,L}^{\frac{6-\mu}{5-\mu}})$ and problem \eqref{CCE} has a nontrivial solution.

If $\lambda\in(\lambda_{j}, \lambda_{j+1})$ for some $j\in\mathbb{N}$, we define
$$
z_{\varepsilon}=u_{\varepsilon}-\sum_{i=1}^{n}\left(\int_{\Omega}u_{\varepsilon}e_{i}dx\right)e_{i},
$$
then,
$$
\mathbb{G}_{j,\varepsilon}=\mathbb{Y}_{j}\oplus \mathbb{R}u_{\varepsilon}=\mathbb{Y}_{j}\oplus \mathbb{R}z_{\varepsilon}.
$$

By Lemma \ref{LK}, we get that $J_{\lambda}$ has the geometric structure required by the Linking Theorem (see [\cite{Rph1}, Theorem 5.3]). Thus we can define the Linking critical level $c_L$ of $J_{\lambda}$ as in \eqref{LL} and
$$
c_L\leq\max\limits_{u\in V}J_{\lambda}(u)\leq\max\limits_{u\in \mathbb{G}_{j,\varepsilon}}J_{\lambda}(u).
$$
On the other hand, we note that for any $u\in H_{0}^{1}(\Omega)\backslash\{0\}$
$$
\max\limits_{t\geq0}J_{\lambda}(tu)=\frac{5-\mu}{12-2\mu}\left(\frac{|\nabla u|_{2}^{2}-\lambda|u|_{2}^{2}}
{\|u\|_{NL}^{2}}\right)^{\frac{6-\mu}{5-\mu}}.
$$
As the same arguments in Section 4, we have
$$\aligned
c_L&\leq\max\limits_{u\in \mathbb{G}_{j,\varepsilon},t\geq0}J_{\lambda}(tu)\\
&=\max\limits_{u\in \mathbb{G}_{j,\varepsilon}}\frac{5-\mu}{12-2\mu}\left(\frac{|\nabla u|_{2}^{2}-\lambda|u|_{2}^{2}}
{\|u\|_{NL}^{2}}\right)^{\frac{6-\mu}{5-\mu}}\\
&<\frac{5-\mu}{12-2\mu}S_{H,L}^{\frac{6-\mu}{5-\mu}}.
\endaligned
$$

Therefore, the Linking Theorem and Lemma \ref{ConPro} yield that problem \eqref{CCE} admits a
solution $u\in H_{0}^{1}(\Omega)$ with critical value $c_L\geq\alpha$. Since $\alpha>0=J_{\lambda}(0)$, we deduce that $u$ is not identically zero.

\textbf{Case 2 } $\lambda_{1}<\lambda_{\ast}$

In this case, we only consider $\lambda\in(\lambda_{j}, \lambda_{j+1})$ for some $j\in\mathbb{N}$ and $\lambda>\lambda_{\ast}$. We can argue as in the last part of Case 1. In this way we get that for any
$\lambda>\lambda_{\ast}$ different from an eigenvalue of $-\Delta$, problem \eqref{CCE} admits a
solution $u\in H_{0}^{1}(\Omega)$ with critical value $c_L\geq\alpha$ and $u$ is not identically zero.

\section{Nonexistence}
 In this Section, we discuss nonexistence of solutions for \eqref{CCE} by using Poho\u{z}aev identity. Firstly, we are going to show that the solutions for equation \eqref{CCE} possess some regularity which will be used to prove the Poho\v{z}aev identity.

\begin{lem}\label{REG}
If $N\geq3$, $\lambda<0$ and $u\in H^{1}(\Omega)$ solves \eqref{CCE},
then $u\in W_{loc}^{2,p}(\Omega)$ for any $p\geq1$.
\end{lem}
\begin{proof}

 Denote by $H=K=|u|^{2_{\mu}^{\ast}-1}=|u|^{\frac{N-\mu+2}{N-2}}$, then $H, K\in L^{\frac{2N}{N-\mu+2}}(\Omega) $.
Using Proposition 3.2 of \cite{MS2}, we know $u\in L^{p}(\Omega)$ for every $p\in[2,\frac{2N^{2}}{(N-\mu)(N-2)})$. Moreover, there exists a constant $C_p$ independent of $u$ such that
$$
\left(\int_{\Omega}|u|^{p}dx\right)^{\frac1p}\leq C_p\left(\int_{\Omega}|u|^{2}dx\right)^{\frac12}.
$$
Thus, $|u|^{2_{\mu}^{\ast}}\in L^{q}(\Omega)$ for every $q\in[\frac{2(N-2)}{2N-\mu},\frac{2N^{2}}{(N-\mu)(2N-\mu)})$. Since $\frac{2(N-2)}{2N-\mu}<\frac{N}{N-\mu}<\frac{2N^{2}}{(N-\mu)(2N-\mu)}$, we have
$\int_{\Omega}\frac{|u|^{2_{\mu}^{\ast}}}{|x-y|^{\mu}}dy\in L^{\infty}(\Omega)$, and so
$$
|-\Delta u-\lambda u|\leq C|u|^{\frac{N-\mu+2}{N-2}}.
$$
By the classical bootstrap method for subcritical local problems in bounded domains, we deduce that $u\in W_{loc}^{2,p}(\Omega)$ for any $p\geq1$.
\end{proof}

\begin{Prop}\label{Poh} If $N\geq3$, $\lambda<0$ and $u\in H^{1}(\Omega)$ solves \eqref{CCE},
then the following equality holds
$$
\frac{1}{2}\int_{\partial\Omega}(x\cdot\nu)|\nabla u|^{2}ds+\frac{N-2}{2}\int_{\Omega} |\nabla u|^{2}dx=\frac{2N-\mu}{22_{\mu}^{\ast}}\int_{\Omega}
\int_{\Omega}\frac{|u(x)|^{2_{\mu}^{\ast}}|u(y)|^{2_{\mu}^{\ast}}}{|x-y|^{\mu}}dxdy+\frac{\lambda N}{2}\int_{\Omega} |u|^{2}dx,
$$
where $\nu$ denotes the unit outward normal to $\partial\Omega$.
\end{Prop}
\begin{proof}
Since $u$ is a solution of \eqref{CCE} and Lemma \ref{REG}, then $u$ satisfies
\begin{equation}
-\Delta u
=\big(\int_{\Omega}\frac{|u|^{2_{\mu}^{\ast}}}{|x-y|^{\mu}}dy\big)|u|^{2_{\mu}^{\ast}-2}u+\lambda u,
\end{equation}
then
\begin{equation}
-\int_{\Omega}(x\cdot\nabla u)\Delta udx=\int_{\Omega}(x\cdot\nabla u)(\int_{\Omega}\frac{|u|^{2_{\mu}^{\ast}}}{|x-y|^{\mu}}dy)|u|^{2_{\mu}^{\ast}-1}dx+
\lambda \int_{\Omega}(x\cdot\nabla u)udx.
\end{equation}
Calculating the first term on the right side, we know
\begin{equation}\aligned
\int_{\Omega}(x\cdot\nabla u(x))&(\int_{\Omega}\frac{|u(y)|^{2_{\mu}^{\ast}}}{|x-y|^{\mu}}dy)|u(x)|^{2_{\mu}^{\ast}-1}dx\\
&=-\int_{\Omega}u(x)\nabla(x\int_{\Omega}\frac{|u(y)|^{2_{\mu}^{\ast}}}{|x-y|^{\mu}}dy|u(x)|^{2_{\mu}^{\ast}-1})dx\\
&=-\int_{\Omega}u(x)(N\int_{\Omega}\frac{|u(y)|^{2_{\mu}^{\ast}}}{|x-y|^{\mu}}dy|u(x)|^{2_{\mu}^{\ast}-1}\\
&\hspace{0.5cm}+(2_{\mu}^{\ast}-1)|u(x)|^{2_{\mu}^{\ast}-2}x\cdot\nabla u(x)\int_{\Omega}\frac{|u(y)|^{2_{\mu}^{\ast}}}{|x-y|^{\mu}}dy\\
&\hspace{1.0cm}+|u(x)|^{2_{\mu}^{\ast}-1}\int_{\Omega}(-\mu)x\cdot(x-y)\frac{|u(y)|^{2_{\mu}^{\ast}}}{|x-y|^{\mu+2}}dy)dx\\
&=-N\int_{\Omega}\int_{\Omega}\frac{|u(x)|^{2_{\mu}^{\ast}}|u(y)|^{2_{\mu}^{\ast}}}{|x-y|^{\mu}}dxdy\\
&\hspace{0.5cm}-(2_{\mu}^{\ast}-1)\int_{\Omega}x\cdot\nabla u(x)\int_{\Omega}\frac{|u(y)|^{2_{\mu}^{\ast}}}{|x-y|^{\mu}}dy|u(x)|^{2_{\mu}^{\ast}-1}dx\\
&\hspace{1.0cm}+\mu\int_{\Omega}\int_{\Omega}x\cdot(x-y)\frac{|u(y)|^{2_{\mu}^{\ast}}}{|x-y|^{\mu+2}}|u(x)|^{2_{\mu}^{\ast}}dydx.
\endaligned
\end{equation}
This implies that
$$\aligned
2_{\mu}^{\ast}\int_{\Omega}(x\cdot\nabla u(x))&(\int_{\Omega}\frac{|u(y)|^{2_{\mu}^{\ast}}}{|x-y|^{\mu}}dy)|u(x)|^{2_{\mu}^{\ast}-1}dx\\
&=-N\int_{\Omega}\int_{\Omega}\frac{|u(x)|^{2_{\mu}^{\ast}}|u(y)|^{2_{\mu}^{\ast}}}{|x-y|^{\mu}}dxdy\\
&\hspace{1.0cm}+\mu\int_{\Omega}\int_{\Omega}x\cdot(x-y)\frac{|u(y)|^{2_{\mu}^{\ast}}}{|x-y|^{\mu+2}}|u(x)|^{2_{\mu}^{\ast}}dydx,
\endaligned
$$
similarly,
$$\aligned
2_{\mu}^{\ast}\int_{\Omega}(y\cdot\nabla u(y))(\int_{\Omega}&\frac{|u(x)|^{2_{\mu}^{\ast}}}{|x-y|^{\mu}}dx)|u(y)|^{2_{\mu}^{\ast}-1}dy\\
&=-N\int_{\Omega}\int_{\Omega}\frac{|u(y)|^{2_{\mu}^{\ast}}|u(x)|^{2_{\mu}^{\ast}}}{|x-y|^{\mu}}dydx\\
&\hspace{0.5cm}+\mu\int_{\Omega}\int_{\Omega}y\cdot(y-x)\frac{|u(x)|^{2_{\mu}^{\ast}}}{|x-y|^{\mu+2}}|u(y)|^{2_{\mu}^{\ast}}dxdy
\endaligned
$$
and consequently, we get
\begin{equation}\aligned
\int_{\Omega}(x\cdot\nabla u(x))(\int_{\Omega}\frac{|u(y)|^{2_{\mu}^{\ast}}}{|x-y|^{\mu}}dy)|u(x)|^{2_{\mu}^{\ast}-1}dx
=\frac{\mu-2N}{22_{\mu}^{\ast}}\int_{\Omega}\int_{\Omega}\frac{|u(x)|^{2_{\mu}^{\ast}}|u(y)|^{2_{\mu}^{\ast}}}{|x-y|^{\mu}}dxdy.
\endaligned
\end{equation}
Since
\begin{equation}
\int_{\Omega}(x\cdot\nabla u)udx=-\frac{N}{2}\int_{\Omega}u^{2}dx
\end{equation}
and
\begin{equation}\aligned
\int_{\partial\Omega}(x\cdot\nu)|\nabla u|^{2}ds=(2-N)\int_{\Omega}|\nabla u|^{2}dx+2\int_{\Omega}(x\cdot\nabla u)\Delta udx.
\endaligned
\end{equation}
From the equalities above, we know  the result holds.
\end{proof}

\textbf{Proof of Theorem \ref{NEXS}.}
We assume that $u$ is a nontrivial solution of \eqref{CCE}, then we have
$$
\int_{\Omega}|\nabla u|^{2}dx=\int_{\Omega}\int_{\Omega}\frac{|u(x)|^{2_{\mu}^{\ast}}|u(y)|^{2_{\mu}^{\ast}}}{|x-y|^{\mu}}dxdy+\lambda\int_{\Omega}u^{2}dx.
$$
From Proposition \ref{Poh}, we can obtain
$$
\int_{\partial\Omega}(x\cdot\nu)|\nabla u|^{2}ds=2\lambda\int_{\Omega} |u|^{2}dx.
$$
Since $\Omega$ is strictly star-shaped with respect to the origin in $\mathbb{R}^N$, then $x\cdot\nu>0$. Thus, we obtain $u\equiv0$ from $\lambda<0$. Which is a contradiction.


\begin{thebibliography}{99}
\bibitem{AC}
\newblock N. Ackermann,
\newblock \emph{ On a periodic Schr\"{o}dinger equation with
nonlocal superlinear part},
\newblock Math. Z., \textbf{248}(2004), 423--443.


\bibitem{ACTY}
\newblock C.O. Alves, D. Cassani, C. Tarsi \& M. Yang,
\newblock \emph{ Existence and concentration of ground state solutions for a critical nonlocal Schr\"{o}dinger equation in $\R^2$},
\newblock J. Differential Equations, Doi:10.1016/j.jde.2016.04.021

\bibitem{ANY}
\newblock C.O. Alves,  A. B. N\'obrega \& M. Yang,
\newblock \emph{ Multi-bump solutions for Choquard equation with deepening potential well},
\newblock Calc. Var. Partial Differential Equations, Doi:10.1007/s00526-016-0984-9.


\bibitem{AY1}
 \newblock C.O. Alves \& M. Yang,
 \newblock \emph{Multiplicity and concentration behavior of solutions for a quasilinear Choquard equation via penalization method},
 \newblock  Proc. Roy. Soc. Edinburgh Sect. A,  \textbf{146}(2016),23--58.

\bibitem{AY2}
 \newblock C.O. Alves \& M. Yang,
 \newblock \emph{ Existence of semiclassical ground state solutions for a generalized Choquard equation}
 \newblock J.
Differential Equations. \textbf{257,}(2014),4133--4164.

\bibitem{BCSS}
\newblock B. Barrios, E. Colorado, R. Servadei \& F. Soria,
\newblock \emph{A critical fractional equation with concave-convex power nonlinearities},
\newblock Ann. Inst. H. Poincar\'{e} Anal. Non Lin\'{e}aire, \textbf{ 32} (2015), 875--900.



\bibitem{BJS}
\newblock B. Buffoni, L. Jeanjean \& C.A. Stuart,
\newblock \emph{ Existence of a nontrivial solution to a strongly
indefinite semilinear equation},
\newblock Proc. Amer. Math. Soc., \textbf{119}(1993), 179--186.



\bibitem{BL1}
\newblock H. Br\'{e}zis \& E. Lieb,
\newblock \emph{A relation between pointwise convergence of functions and convergence of functionals,}
\newblock Proc.Amer.Math.Soc., \textbf{88} (1983), 486--490.

\bibitem{BN}
\newblock H. Br\'{e}zis \& L. Nirenberg,
\newblock \emph{Positive solutions of nonlinear elliptic equations involving critical Sobolev exponents,}
\newblock Comm. Pure Appl. Math., \textbf{36} (1983), 437--477.


\bibitem{CP}
\newblock D. Cao \& S. Peng,
\newblock \emph{A note on the sign-changing solutions
to elliptic problems with critical Sobolev and Hardy terms},
\newblock J.
Differential Equations. \textbf{ 193} (2003), 424--434.

\bibitem{CFP}
\newblock
A. Capozzi, D. Fortunato \& G. Palmieri,
\newblock \emph{An existence result for nonlinear elliptic problems involving critical Sobolev exponent},
\newblock Ann. Inst. H. Poincar\'{e} Anal. Non Lin\'{e}aire, \textbf{2} (1985), 463--470.

\bibitem{CFS}
\newblock
G. Cerami, D. Fortunato \& M. Struwe,
\newblock \emph{Bifurcation and multiplicity results for nonlinear elliptic problems involving critical Sobolev exponents},
\newblock Ann. Inst. H. Poincar\'{e} Anal. Non Lin\'{e}aire, \textbf{1} (1984), 341--350.

\bibitem{CSS}
\newblock
G. Cerami, S. Solimini \& M. Struwe,
\newblock \emph{Some existence results for superlinear elliptic boundary value problems involving critical exponents},
\newblock J. Funct. Anal., \textbf{69} (1986), 289--306.


\bibitem{CS}
\newblock M. Clapp \& D. Salazar,
\newblock \emph{ Positive and sign changing solutions to a nonlinear
Choquard equation,}
\newblock J. Math. Anal. Appl. \textbf{407} (2013), 1--15.

\bibitem{CCS1}
\newblock
S. Cingolani, M. Clapp \& S. Secchi,
\newblock \emph{Multiple solutions to a magnetic nonlinear Choquard equation},
\newblock Z.
Angew. Math. Phys., \textbf{63} (2012), 233--248.

\bibitem{FG}
\newblock A. Ferrero\&  F. Gazzola,
\newblock \emph{Existence of solutions for
singular critical growth semilinear elliptic equations},
\newblock J.
Differential Equations. \textbf{177} (2001), 494--552.

\bibitem{GS}
\newblock M. Ghimenti \& J. Van Schaftingen,
\newblock \emph{ Nodal solutions for the Choquard equation},
\newblock J. Funct. Anal. 271 (2016), no. 1, 107--135.

\bibitem{GMS}
\newblock M. Ghimenti, V. Moroz \& J. Van Schaftingen,
\newblock \emph{ Least Action nodal solutions for the quadratic Choquard equation},
\newblock arXiv:1511.04779v1
\bibitem{GY}
\newblock N.  Ghoussoub \& C. Yuan,
\newblock \emph{Multiple solutions for quasilinear
 PDEs involving the critical Sobolev and
 Hardy exponent},
\newblock Trans. Amer. Math. Soc. \textbf{352}(2000),5703--5743.

\bibitem{J}
\newblock E. Janneli,
\newblock \emph{The role played by space dimension in
elliptic critical problems},
\newblock J. Differential Equations. \textbf{156}
(1999), 407--426.

\bibitem{L1}
\newblock E. Lieb,
\newblock \emph{Existence and uniqueness of the minimizing solution of Choquard's nonlinear equation},
\newblock  Studies
in Appl. Math., \textbf{57}(1976/77), 93--105.

\bibitem{LL}
\newblock E. Lieb \& M. Loss, \newblock "Analysis,"
\newblock \emph{Gradute Studies in Mathematics}, AMS,
Providence, Rhode island, 2001.

\bibitem{Ls}
\newblock P. Lions,
\newblock \emph{The Choquard equation and related questions},
 \newblock Nonlinear Anal., \textbf{4}(1980), 1063--1072.

\bibitem{ML}
\newblock L. Ma \& L. Zhao,
\newblock \emph{Classification of positive solitary solutions of the nonlinear Choquard equation},
\newblock Arch.
Ration. Mech. Anal., \textbf{195}(2010), 455--467.

\bibitem{MS1}
\newblock V. Moroz \& J. Van Schaftingen,
 \newblock \emph{Ground states of nonlinear Choquard equations: Existence, qualitative properties and decay asymptotics},
\newblock J. Funct. Anal., \textbf{265}(2013), 153--184.

\bibitem{MS2}
\newblock V. Moroz \& J. Van Schaftingen,
 \newblock \emph{Existence of groundstates for a class of nonlinear Choquard equations},
 \newblock  Trans. Amer. Math. Soc. doi:10.1090/S0002-9947-2014-06289-2

\bibitem{MS3}
\newblock V. Moroz \& J. Van Schaftingen,
\newblock \emph{Semi-classical states for the Choquard equation},
\newblock Calc. Var. Partial Differential Equations, \textbf{ 52} (2015), 199--235.

\bibitem{MS4}
\newblock V. Moroz \& J. Van Schaftingen,
\newblock \emph{Groundstates of nonlinear Choquard equations: Hardy-Littlewood-Sobolev critical exponent},
\newblock Commun. Contemp. Math., \textbf{17} (2015), 1550005, 12 pp.


\bibitem{P1}
\newblock S. Pekar,
\newblock \emph{Untersuchung\"{u}ber die Elektronentheorie der Kristalle},
\newblock Akademie Verlag, Berlin, 1954.

\bibitem{Pe}
\newblock R. Penrose,
\newblock \emph{On gravity's role in quantum state reduction,}
\newblock  Gen. Relativ. Gravitat., \textbf{28} (1996), 581--600.

\bibitem{Rph1} P. Rabinowitz, Minimax methods in critical point theory with applications to differential equations. CBMS Reg. Conf. Ser. Math., 65, American Mathematical Society, Providence, RI, 1986.

\bibitem{S}
\newblock S. Secchi,
 \newblock \emph{A note on Schr\"odinger-Newton systems with decaying electric potential},
  \newblock Nonlinear Anal., \textbf{72} (2010), 3842--3856.



\bibitem{SZ}
\newblock M. Schechter \& W. M. Zou,
 \newblock \emph{On the Br\'{e}zis-Nirenberg problem,}
 \newblock Arch. Ration. Mech. Anal.,\textbf{197} (2010), 337--356.




\bibitem{SV2}
\newblock R. Servadei \& E. Valdinoci,
\newblock \emph{The Brezis-Nirenberg result for the fractional Laplacian},
\newblock Trans. Amer. Math. Soc., \textbf{367} (2015), 67--102.

\bibitem{Ta}
\newblock J. Tan,
\newblock \emph{The Brezis-Nirenberg type problem involving the square root of the Laplacian},
\newblock Calc. Var. Partial Differential Equations, \textbf{ 42} (2011), 21--41.

\bibitem{WW}
\newblock J. Wei \& M. Winter,
\newblock \emph{Strongly Interacting Bumps for the Schr\"odinger-Newton Equations},
\newblock J. Math. Phys., \textbf{50} (2009), 012905, 22 pp.


\bibitem{Wi} M. Willem, Minimax Theorems,  Progress in Nonlinear Differential Equations and their Applications, 24. Birkh\"{a}user Boston, Inc., Boston, MA, 1996.

\bibitem{Wi2} M. Willem, Functional analysis,  Fundamentals and applications, in: Cornerstones, vol. XIV, Birkh\"{a}user/Springer, New York, 2013.

\bibitem{YD}
\newblock
M. Yang \& Y. Ding,
\newblock \emph{ Existence of solutions for singularly perturbed Schr\"odinger equations with nonlocal part},
\newblock Comm. Pure Appl. Anal.,
\textbf{12}(2013), 771--783.



\end{thebibliography}
\end{document}